\documentclass[12pt,twoside,letterpaper]{amsart}
\usepackage{amsmath,amssymb,url,graphicx,amsthm,fullpage}
\theoremstyle{definition}
\numberwithin{equation}{section}
\newtheorem{theorem}[equation]{Theorem}
\newtheorem{lemma}[equation]{Lemma}
\newtheorem{proposition}[equation]{Proposition}
\newtheorem{corollary}[equation]{Corollary}

\newtheorem{definition}[equation]{Definition}
\newtheorem{example}[equation]{Example}

\newtheorem{remark}[equation]{Remark}

\newcommand{\Rep}{\operatorname{\mathsf{Rep}}}

\newcommand{\KK}{\mathbb{K}}

\newcommand{\ZZ}{\mathbb{Z}}

\newcommand{\NN}{\mathbb{N}}

\newcommand{\OO}{\mathcal{O}}

\newcommand{\Hom}{\operatorname{Hom}}
\newcommand{\Ind}{\operatorname{Ind}}
\newcommand{\Ext}{\operatorname{Ext}}

\newcommand\id{\operatorname{id}}
\newcommand\gr{\operatorname{gr}}
\newcommand\Spec{\operatorname{Spec}}
\newcommand\hp{\hphantom{x}}

\newcommand{\Par}{\operatorname{Par}}
\newcommand{\Stab}{\operatorname{Stab}}

\newcommand\F{\operatorname{F}}
\newcommand\J{\mathcal{J}}
\newcommand\N{\NN}
\newcommand\Z{\ZZ}

\newcommand\Prim{\operatorname{Pr}}
\newcommand\VA{\operatorname{V}}
\newcommand\m{\mathfrak{m}}
\newcommand\W{{\textbf{W}}}
\newcommand\K{\mathbb K}
\newcommand\ad{\operatorname{ad}}
\newcommand\Dim{\operatorname{Dim}}

\begin{document}

\title{Poisson traces and $D$-modules on Poisson varieties}

\author{Pavel Etingof and Travis Schedler \\
\\
with an appendix by Ivan Losev}
\date{April 10, 2009}
\maketitle
\begin{abstract} 
  To every Poisson algebraic variety $X$ over an algebraically closed
  field of characteristic zero, we canonically attach a right
  $D$-module $M(X)$ on $X$. If $X$ is affine, solutions of $M(X)$ in
  the space of algebraic distributions on $X$ are Poisson traces on
  $X$, i.e., distributions invariant under Hamiltonian flow. When $X$
  has finitely many symplectic leaves, we prove that $M(X)$ is
  holonomic.  Thus, when $X$ is affine and has finitely many
  symplectic leaves, the space of Poisson traces on $X$ is
  finite-dimensional. More generally, to any morphism $\phi: X \rightarrow Y$ and any
  quasicoherent sheaf of Poisson modules $N$ on $X$, we attach a right
  $D$-module $M_\phi(X, N)$ on $X$, and prove that it is holonomic if
  $X$ has finitely many symplectic leaves, $\phi$ is finite, and $N$
  is coherent.  
  
  As an application, we deduce that noncommutative filtered algebras,
  for which the associated graded algebra is finite over its center whose
  spectrum has finitely many symplectic leaves, have finitely many
  irreducible finite-dimensional representations. The appendix, by
  Ivan Losev, strengthens this to show that in such algebras, there
  are finitely many prime ideals, and they are all primitive. This
  includes symplectic reflection algebras.

  Furthermore, we describe explicitly (in the settings of affine
  varieties and compact $C^\infty$-manifolds) the finite-dimensional
  space of Poisson traces on $X$ when $X=V/G$, where $V$ is symplectic
  and $G$ is a finite group acting faithfully on $V$.
\end{abstract}
\bigskip

\centerline{\textbf{Table of Contents}}

$\hspace{15mm}${\small \parbox[t]{135mm}{
\noindent
\hp1.{ $\;\,$} Introduction\\ \hp2.{ $\;\,$} The
  $D$-module $M_\phi(X)$\\
\hp3.{ $\;\,$} The holonomicity theorem and applications\\
\hp4.{ $\;\,$} The structure of $M(X)$ when $X$ has finitely many
  symplectic leaves\\
\hp A.{ $\;\,$} Appendix: Prime and primitive ideals, by Ivan
  Losev  \\
\hp B.{ $\;\,$} Appendix: Possible generalizations} } \bigskip

\section{Introduction}

Let $A$ be a Poisson algebra. In this paper, we always work over an
algebraically closed field $\mathbb K$ of characteristic zero.  A {\it
  Poisson trace}
on $A$ is a linear functional $T: A\to \mathbb K$ such that
$T(\lbrace{a,b\rbrace})=0$ for all $a$ and $b$ in $A$, i.e., a Lie
algebra character of $A$.  Thus, the space of Poisson traces of $A$
coincides with the zeroth Lie algebra cohomology
$H^0_{\mathrm{Lie}}(A,A^*)$, and is dual to the zeroth Lie algebra homology
$H_0^{\mathrm{Lie}}(A,A)=A/\lbrace{ A,A\rbrace}$, which is also the {\it
  zeroth Poisson homology} $HP_0(A)$.

Unfortunately, in spite of the simplicity of their definition, Poisson
traces remain rather poorly understood.  This is partly because this
definition is non-local with respect to $\Spec(A)$, so the
machinery of algebraic geometry cannot be directly applied.

In this paper, we develop a general theory of $D$-modules on Poisson
varieties, which resolves this issue.  This theory canonically
attaches to every Poisson variety $X$ a certain right $D$-module
$M(X)$ on $X$, which is local with respect to $X$, and whose zeroth
cohomology is $HP_0(\OO_X)$.

More generally, in \S 2 we attach a right $D$-module $M_\phi(X)$ on
$X$ to every Poisson variety $X$ together with a morphism $\phi: X\to
Y$ to another variety $Y$. If $X=Y$ and $\phi=\id$, then
$M_\phi(X)=M(X)$. Still more generally, given a quasicoherent sheaf of
Poisson modules $N$ on $X$ (the notion dates to at least
\cite{RVWsg}; we recall the definition in \S \ref{pmdefs} below), we
define the right $D$-module $M_\phi(X, N)$ on $X$; if $N=\OO_X$, it
coincides with $M_\phi(X)$.

We show that, if $X$ and $Y$ are affine, the underived direct image of
$M_\phi(X)$ under the map of $X$ to a point is
$\OO_X/\lbrace{\OO_Y,\OO_X\rbrace}$; more generally, for $M_\phi(X,N)$ it is
$N/\lbrace{\OO_Y, N\rbrace}$. In particular, $HP_0(\OO_X)$ is the underived
  direct image of $M(X)$.

  Next, in \S 3, we show that, if $X$ has finitely many symplectic
  leaves and $\phi$ is a finite morphism, the $D$-module $M_\phi(X)$
  is holonomic, and so is $M_\phi(X, N)$ for a coherent sheaf of
  Poisson modules $N$. Since the direct image of a holonomic
  $D$-module is holonomic, this implies the following results.

\begin{theorem}\label{th1}
  Let $X$ be an affine Poisson variety, $Y$ be another affine variety,
  and $\phi: X\to Y$ be a morphism. If $X$ has finitely many
  symplectic leaves and $\phi$ is finite, then the space
  $\OO_X/\lbrace{\OO_Y,\OO_X\rbrace}$ is finite-dimensional.
  More generally, $N/\{\OO_Y, N\}$ is finite-dimensional for any 
  coherent sheaf of Poisson modules $N$ over $\OO_X$.
\end{theorem}

\begin{corollary}\label{co1} 
  If $X$ is an affine Poisson variety which has finitely many
  symplectic leaves, then $HP_0(\OO_X)$ is finite-dimensional.
\end{corollary}

\begin{corollary}\label{co2}
  Let $X$ be an affine Poisson variety with finitely many symplectic
  leaves (for example, a symplectic variety) and $G$ be a finite group
  of Poisson automorphisms of $X$.  Then,
  $\OO_X/\lbrace{\OO_{X/G},\OO_X\rbrace}$ is finite-dimensional.  In
  particular, the subspace of $G$-invariants, $HP_0(\OO_{X/G})$, is
  finite-dimensional.
\end{corollary}

Note that the last statement of Corollary \ref{co2} also follows from
Corollary \ref{co1}, since, if $X$ has finitely many symplectic
leaves, so does $X/G$.

Corollary \ref{co2} in the case when $X$ is a symplectic vector space
$V$ and $G$ acts linearly was proved in \cite[\S 7]{BEG} (the last
statement of the corollary in this case was previously conjectured by
Alev and Farkas in \cite{AlFa}). In fact, the proof in \cite{BEG} is
based on showing that the Fourier transform of the $D$-module
$M_\phi(V)$ (where $\phi: V\to V/G$ is the tautological map) is
generically defined by a holonomic system of differential equations,
and the method of this paper is, essentially, an extension of the
method of \cite{BEG} to general varieties.

Our results imply the following theorem in noncommutative algebra.

  Let $A$ be a nonnegatively filtered associative algebra, such that
  $A_0:=\gr(A)$ is a finitely generated module over its center $Z$, 
  and let $X:={\rm Spec}(Z)$ be the corresponding Poisson variety.
  Assume that $X$ has finitely many symplectic
  leaves. More precisely, we assume that for some integer $d>0$ there is a filtered linear lift 
$g: Z\hookrightarrow A$ of the inclusion $Z\hookrightarrow A_0$ such that 
$[g(Z^i),F_jA]\subset F_{i+j-d}A$. In this case, 
$Z$ acquires a natural homogeneous Poisson bracket $\lbrace,\rbrace$ 
of degree $-d$, independent on the choice of $g$, such that for $a\in Z^i$, $b\in Z^j$, 
$\lbrace a,b\rbrace$ is the image of $[g(a),g(b)]$ in 
$F_{i+j-d}(A)$, and we assume that $X$ has finitely many symplectic leaves with 
respect to this bracket. In particular, if $A_0$ is commutative, the requirement is 
that $[F_iA,F_jA]\subset F_{i+j-d}(A)$ (the number $d$ is then the maximal integer with this 
property.)\footnote{For convenience, let us forbid the trivial situation when $A$ is finite-dimensional.} 

\begin{theorem}\label{th2}
Under the above assumption, the space $HH_0(A)=A/[A,A]$ is finite-dimensional.  In
  particular, $A$ has finitely many irreducible finite-dimensional
  representations.  
\end{theorem}

\begin{example} Let $V$ be a symplectic vector space and $G \subset
  Sp(V)$ a finite subgroup.  For every conjugation-invariant function
  $c$ on the set of symplectic reflections of $G$, one may form the
  symplectic reflection algebra $A = H_{1,c}$ \cite{EGch} (deforming
  the algebra $\mathrm{Weyl}(V) \rtimes \KK[G]$), equipped with its
  natural filtration. The associated graded algebra, $A_0 = \gr
  H_{1,c}$ is isomorphic to $\operatorname{Sym} V \rtimes \KK[G]$, and
  hence is finite over its center $(\operatorname{Sym} V)^G$ whose
  spectrum $V/G$ is a Poisson variety with finitely many symplectic
  leaves (with Poisson bracket of degree $-2$, i.e., $d=2$). 
Hence, $A$ has finitely many irreducible finite-dimensional
  representations.\footnote{For rational Cherednik algebras, this is
    well-known: all finite-dimensional representations belong to
    category $\OO$, so the number of irreducible finite-dimensional
    representations is dominated by the number of irreducible
    representations of $G$.}
\end{example}

\begin{remark}\label{hbrem}
  Theorem \ref{th2} extends (with essentially the same proof) to the
  case of \emph{deformation quantizations} of $A_0$. That is, we
  replace $A$ by a $\KK[[\hbar]]$-algebra $A_\hbar$ such that $A_\hbar
  \cong A_0[[\hbar]] = \{\sum_{i \geq 0} a_i \hbar^i: a_i \in A_0\}$
  as an $\KK[[\hbar]]$-module, $A_\hbar/(\hbar) \cong A_0$ as an
  algebra, $A_0$ is a finitely generated module over its center $Z$,
  and $\Spec Z$ is a Poisson variety with finitely many symplectic
  leaves.\footnote{More precisely, similarly to above, we assume that
    for some integer $d>0$, there is a linear lift $g: Z\to A$ such
    that for any $z\in Z$ and $a\in A$, $[g(z),a]$ is divisible by
    $\hbar^d$ (this condition is equivalent to the condition that the
    center of $A/\hbar^d A$ is free as a $\Bbb
    K[h]/\hbar^d$-module). In this case, $Z$ acquires a natural
    Poisson bracket $\lbrace a,b\rbrace :=\hbar^{-d}[g(a),g(b)]\text{
      mod }\hbar$ independent of the lift, and we assume that $\Spec
    Z$ has finitely many symplectic leaves with respect to this
    bracket.  In the case when $A_0$ is commutative, the requirement
    is that $[A,A]\subset \hbar^d A$, and $d$ is the maximal integer
    with this property.}  Then, $HH_0(A_\hbar[\hbar^{-1}])$ is
  finite-dimensional over $\KK((\hbar))$, and in particular,
  $A_\hbar[\hbar^{-1}]$ has finitely many irreducible
  finite-dimensional representations (over $\KK((\hbar))$).  Note that
  this subsumes the above results, since if $A$ is a filtered
  quantization of $A_0$, then the completed Rees algebra $A_\hbar :=
  \widehat{\bigoplus}_{i \geq 0} (F_i A) \cdot \hbar^i$ is a
  deformation quantization of $A_0$. This setting applies to much more
  general $A_0$, since it need not be graded, and even if it is,
  $\{-,-\}$ is allowed to be homogeneous of nonnegative degree.
\end{remark}

\begin{remark}
  As explained in \S \ref{sche}, these results extend in a
  straightforward manner from the setting of varieties to the
more general setting of (not necessarily reduced) separated schemes of finite 
type over $\mathbb K$.  In \S \ref{fsche} we 
  comment on the case of formal schemes.
\end{remark}

In \S 4 we study the structure of the $D$-module $M(X)$ when $X$ has
finitely many symplectic leaves.  In particular, we determine the
structure of $M(X)$ for $X=V/G$, where $V$ is a symplectic variety,
and $G$ is a finite group of symplectic transformations of $V$.  We
also compute the space of distributions on a compact symplectic
manifold $V$ with an action of a finite group $G$ which are preserved
by $G$-invariant Hamiltonian vector fields, and in particular show
that this space is finite-dimensional.

The appendix, written by Ivan Losev, contains a generalization of
Theorem \ref{th2}, which states that, under the same hypotheses on
$A$, it has finitely many prime ideals, and they are all
primitive. Moreover, the result extends to deformation
quantizations $A_\hbar[\hbar^{-1}]$ as in Remark \ref{hbrem}, as algebras
over $\KK((\hbar))$.  

Finally, in a second appendix, we pose some questions (suggested by
Losev) which would generalize the results of Losev's appendix.

\textbf{Acknowledgments.}  We are grateful to Eric Rains for a
discussion in December 2008 which triggered this work. We thank Dennis
Gaitsgory for clarifying many technical points (in particular, for
supplying the proofs of Lemmas \ref{push} and \ref{semis}), and Ivan
Losev, Dmitry Kaledin, and David Vogan for useful discussions. We are
grateful to Losev for contributing his appendix to this paper. The
first author's work was partially supported by the NSF grant
DMS-0504847. The second author is a five-year fellow of the American
Institute of Mathematics, and was partially supported by the
ARRA-funded NSF grant DMS-0900233.

\section{The $D$-module $M_\phi(X)$}

\subsection{The smooth affine case}

Let $X$ be a smooth affine Poisson algebraic variety, and $b$ be the
corresponding Poisson bivector (i.e., biderivation of $\OO_X$).  Let
$D_X$ denote the algebra of differential operators on $X$.

For any $f \in \OO_X$, define the Hamiltonian vector field $\xi_f \in
\text{Vect}(X) = \operatorname{Der}(\OO_X)$ by
\begin{equation}
\xi_f(g) = \{f, g\}.
\end{equation}
Let $\text{HVect}(X) \subseteq \text{Vect}(X)$ denote the subspace of Hamiltonian vector fields.  

\begin{definition}
  The right $D$-module $M(X)$ attached to $X$ is the
  quotient
\begin{equation}
M(X) := \langle \text{HVect}(X)\rangle \setminus D_X
\end{equation}
of $D_X$ by the right ideal generated by $\text{HVect}(X)$.
\end{definition}
Next, let $Y$ be another affine variety, and $\phi: X \rightarrow Y$ a morphism.
Let $\text{HVect}(X, \phi^* \OO_Y) \subseteq \text{HVect}(X)$ denote the subspace of Hamiltonian vector fields of functions pulled back from $Y$.
\begin{definition}
\begin{equation}
M_\phi(X) := \langle\text{HVect}(X, \phi^*(\OO_Y))\rangle \setminus D_X.
\end{equation}
\end{definition}
Note that $M(X) = M_{\id}(X)$.
\begin{example}
  Let $X$ be a symplectic variety.  Then, we claim that $M(X)$ is
  canonically isomorphic to $\Omega$, the right $D$-module of volume
  forms on $X$.  Indeed, there is a homomorphism
  $\eta: M(X)\to \Omega$ sending $1$ to the canonical volume form on
  $X$ (it exists because the volume is preserved by Hamiltonian vector
  fields).  Since $\Omega$ is irreducible, $\eta$ is surjective, and
  it is easily seen to be injective by considering the associated
  graded morphism.
\end{example}

\begin{remark} 
  It is clear that $M_\phi(X)$ is independent of $Y$ in the following
  sense. Suppose that $\iota: Y\to Y'$ is a closed embedding.  Then,
  $M_\phi(X)=M_{\iota \circ \phi}(X)$.  So, we may assume that $\phi$
  is a dominant morphism (taking the minimal possible $Y$). This
  corresponds to the situation when the morphism
  $\phi^*: \OO_Y\to \OO_X$ is injective.
\end{remark}

\subsection{The singular case} 

Suppose now that $X$ is not
necessarily smooth, but still affine.  Then, while right
$D$-modules are not the same as modules over an algebra of
differential operators, there is
still a canonical right $D$-module, $D_X$, which represents the
functor of global sections on the category of right $D$-modules on
$X$.  

The $D$-module $D_X$ carries an action of $\OO_X$ and a compatible
Lie algebra action of ${\rm Vect}(X)$ by endomorphisms.  
Thus, the above definitions go through verbatim.

Explicitly, if $i: X \to V$ is an embedding of $X$ into a smooth
affine variety $V$, then by Kashiwara's theorem, the functor
$i_*$ defines an equivalence between the category of right
$D$-modules on $X$ and the category of right $D$-modules
on $V$ supported on $X$. Under this equivalence, $D_X$ maps to the
quotient $\langle I_X\rangle \setminus D_V$ of $D_V$ by the right
ideal generated by the defining ideal $I_X\subset \OO_X$ of $X$,
and it is easy to show that $M(X)$ maps to the quotient of $D_V$ by the
right ideal generated by $I_X$ and vector fields on $V$ tangent
to $X$ whose restrictions to $X$ are Hamiltonian. 
We denote this $D$-module on $V$ by $M(X,i)$.
Similarly, given $\phi: X \rightarrow Y$ as before, $M_\phi(X)$ maps
to the quotient of $D_V$ by the right ideal generated 
$I_X$ and vector fields on $V$
tangent to $X$ and specializing at $X$ to $\xi_f$ for some 
$f\in \OO_Y$. We denote this
$D$-module by $M_\phi(X,i)$.

\subsection{Compatibility with algebraic group actions}

Now suppose that, in the setting of the previous subsection, $X$
carries an algebraic action of an affine algebraic group $G$
preserving the Poisson bracket, and that the map $\phi$ is
$G$-invariant (where $G$ acts on $Y$ trivially). Then, it is easy to
see that $M_\phi(X)$ has a natural structure of weakly
$G$-equivariant $D$-module on $X$.  So, if $G$ is finite, $M_\phi(X)$
is naturally a $G$-equivariant $D$-module.

\begin{remark} 
  Note that, since the $D$-module $M_\phi(X)$ does not change when the
  Poisson bracket is rescaled, these statements remain true when $G$
  preserves the Poisson bracket on $X$ only up to scaling.
\end{remark}

\subsection{The behavior of $M_\phi(X)$ under taking quotients by finite groups} 

Let $G$ be a finite group, $X$ be an affine $G$-variety, and
$\pi: X\to X/G$ be the tautological map.  Let $\pi_*^G=\pi_!^G$ be the
\emph{equivariant pushforward} functor from the category of
$G$-equivariant right $D$-modules on $X$ to the category of right
$D$-modules on $X/G$ (it is the composition of taking the direct image
and passing to $G$-invariants).

\begin{lemma}\label{push}
  $\pi_*^G(D_X)=D_{X/G}$.
\end{lemma}

\begin{proof}
  Let $\Ind$ be the induction functor from $\OO$-modules to right
  $D$-modules on any algebraic variety $Y$, i.e.,
  $\Ind(M)=M\otimes_{\OO_Y}D_Y$. Then, $\Ind(\OO_Y)=D_Y$.
  Also, $\Ind$ commutes with direct images, since
  $\Ind\circ \pi_*=\pi_*\circ \Ind$ for all proper morphisms
  $\pi$.  Therefore,
  \[
  \pi_*^G(D_X)=\pi_*^G(\Ind(\OO_X))={\rm
    Ind}(\pi_*(\OO_X)^G)=\Ind(\OO_{X/G})=D_{X/G}. \qedhere 
  \]
\end{proof}
Next, let $X$ be Poisson and $G$ act by Poisson automorphisms.  Let
$Y$ be another affine variety and $\phi: X\to Y$ be a $G$-equivariant
morphism (where $G$ acts trivially on $Y$).

\begin{corollary}
  \label{pushpois} $\pi_*^G(M_\phi(X))=M_\phi(X/G)$.
\end{corollary}

\begin{proof}
  This follows from Lemma \ref{push}, since both $M_\phi(X)$ and
  $M_\phi(X/G)$ are the $D$-modules of coinvariants of Hamiltonian
  vector fields associated to functions pulled back from $Y$.
\end{proof}

\subsection{The case of general varieties} 

Next, we generalize the definition of $M_\phi(X)$ to the case when X
is not necessarily affine. We use the same definition, except where
$\text{HVect}(X)$ is viewed as a subsheaf of the tangent sheaf on $X$,
and similarly $\text{HVect}(X, \phi^* \OO_Y)$ in the case of a
morphism $\phi:X \rightarrow Y$. In particular, the definition is
local, in the following sense:
\begin{lemma}\label{locpm}  If $j: U \to X$ is an open embedding of Poisson
varieties, then
  $M_\phi(U) = j^!M_\phi(X) = j^*M_\phi(X)$ is the restriction of $M_\phi(X)$ to $U$.
\end{lemma}

\subsection{The behavior of $M_\phi(X)$ under closed Poisson embeddings}

\begin{proposition}\label{closedemb}
  Let $i: Z\to X$ be a closed Poisson embedding and $\phi: X\to Y$ be
  a morphism. Then, there is a natural epimorphism
  $\theta_i: M_\phi(X)\to i_*M_\phi(Z)$.
\end{proposition}

\begin{proof}
  By definition, $i_*M_\phi(Z)$ is the quotient of $D_X$ by the right
  ideal generated by regular functions vanishing on $Z$ and vector
  fields on $X$ which are tangent to $Z$ and specialize there to
  Hamiltonian vector fields associated to functions pulled back from
  $Y$.  This ideal contains the Hamiltonian vector fields on $X$
  associated to functions pulled back from $Y$, since they are all
  tangent to $Z$ (by definition of a Poisson embedding).
\end{proof}

\subsection{The direct image of $M_\phi(X)$ to the point}

Let $X$ and $Y$ be affine.

\begin{proposition}\label{hp0} 
  Let $p: X\to pt$ be the projection.  The space
  $\OO_X/\lbrace{\OO_Y,\OO_X\rbrace}$ is the (underived) direct image
  $p_0(M_\phi(X))$ of $M_\phi(X)$ under $p$.  In particular,
  $p_0(M(X))$ is the zeroth Poisson homology $HP_0(\OO_X)$.
\end{proposition}

\begin{proof}
  Let $i :X\to V$ be a closed embedding of $X$ into a smooth affine
  variety $V$.  It is easy to see that
  $\OO_X/\lbrace{\OO_Y,\OO_X\rbrace}=M_\phi(X,i)\otimes_{D_V}\OO_V$,
  which implies the statement.
\end{proof}
 
\begin{remark}\label{pdr}
  Define the Poisson-de Rham homology of an affine Poisson variety $X$
  to be the full direct image $p_*(M(X))$ (the de Rham cohomology of
  $M(X)$), i.e., $HP^{DR}_j(X):=L_jp_0(M(X))$, $j\ge 0$.  This gives a
  new homology theory of affine Poisson varieties (which can be
  extended in an obvious way to non-affine varieties, with homology
  living in degrees $-\dim X\le d\le \dim X$).  If $X$ is symplectic
  of dimension $n$, then
  $HP^{DR}_j(X)=HP_j(\OO_X)=H^{n-j}(X,{\mathbb K})$.\footnote{The
    equality $HP_j(\OO_X)=H^{n-j}(X,{\mathbb K})$ also holds for smooth
    and complex analytic symplectic manifolds (\cite{Br}). This
    implies that for such manifolds, Poisson homology (in particular,
    $HP_0$) need not be finite-dimensional.  For instance, consider
    the complex analytic manifold
    $X=\mathbb C^*\times (\mathbb C\setminus \mathbb Z)\subset \mathbb C^2$ with
    the standard symplectic structure $dz\wedge dw$; its $H^2$ is
    infinite-dimensional, so $\dim HP_0(\OO_X)=\infty$.}
  Furthermore, Proposition \ref{hp0} implies that, for arbitrary
  affine $X$, $HP^{DR}_0(X)=HP_0(\OO_X)$.  However, this is not true
  for higher homology, in general.
\end{remark}

\begin{corollary}\label{fidi}
  If $M_\phi(X)$ is holonomic, then
  $\OO_X/\lbrace{\OO_Y,\OO_X\rbrace}$ is finite-dimensional.
\end{corollary}

\begin{proof} 
  The direct image of a holonomic $D$-module is holonomic (see
  \cite{Kash}).
\end{proof} 

\subsection{The case of Poisson modules} \label{pmdefs}
The above definitions and statements naturally generalize to the context of Poisson $\OO_X$-modules.  We first recall their definition.  

Let $B$ be a Poisson algebra. 
\begin{definition} A $B$-module $M$ is called a Poisson module 
if it is endowed with a Lie algebra action $\lbrace{,\rbrace}:
B\otimes M\to M$ that satisfies the Leibniz rules:
$$ 
\lbrace{a,bm\rbrace}=\lbrace{a,b\rbrace}m+b\lbrace{a,m\rbrace},\
\lbrace{ab,m\rbrace}=a\lbrace{b,m\rbrace}+b\lbrace{a,m\rbrace}.
$$
\end{definition}

\begin{example}\label{pmexams}\begin{enumerate}
\item $B$ is a Poisson module over itself. 
More generally, if $B$ is contained in a 
larger Poisson algebra $C$, then $C$ is a Poisson module over
$B$. 

\item\label{pmdmex} Let $B=\OO_X$, where $X$ is a symplectic variety.  Then a
  Poisson module over $B$ is the same thing as a (left) $D$-module on
  $X$.  More generally, if $X$ is an arbitrary affine Poisson variety,
  then a Poisson module over $\OO_X$ is an $\OO_X$-module $M$ with 
  flat connections along symplectic leaves (which glue together to give
  a global action of Hamiltonian vector fields).  In particular, $D_X$
  is always a Poisson module over $\OO_X$, and for smooth $X$, if $M$
  is a left $D$-module on $X$, then $M$ is naturally a Poisson module
  over $\OO_X$ (but not conversely).

\item \label{mbexam} Let $A_\hbar$ be a (not necessarily commutative)
  flat deformation of a commutative algebra $A_0$ over
  $\KK[[\hbar]]$. Let $M_\hbar$ be an $A_\hbar$-bimodule which is
  topologically free over $\KK[[\hbar]]$ and almost symmetric,
  i.e. such that for any $a\in A_\hbar$ and $m\in M_\hbar$, $am \equiv
  ma \pmod \hbar$. Then $M_0=M_\hbar/\hbar M_\hbar$ is naturally a
  Poisson module over $A_0$.

\item \label{a0z0exam} Let $A_0$ be an associative algebra,
and $A_\hbar$ a flat deformation of $A_0$ over $\KK[[\hbar]]$.
Let $Z$ be the center of $A_0$, and assume that $Z$ admits a linear lift $g: Z\to A$ 
which is central modulo $\hbar^d$. Then $Z$ inherits a Poisson structure 
from the deformation, as in Remark \ref{hbrem}, and the Hochschild homology and cohomology
of $A_0$ are Poisson modules over $Z$. In particular, $A_0/[A_0,A_0]$ is a Poisson module over $Z$.
Indeed, let ${\bold d}: Z=HH^0(A_0)\to HH^1(A_0)$ be the first  
nonzero differential in the spectral sequence computing the Hochschild cohomology of $A$
(i.e., ${\bold d}={\bold d}_d$). Then the Lie algebra action of $Z$ on the Hochschild 
homology and cohomology of $A_0$ is given by the formula 
$\lbrace{z,a\rbrace}=L_{{\bold d}z}(a)$, where $L_{{\bold d}z}$ is the Lie
derivative with respect to the outer derivation ${\bold d}z$ of $A_0$. 

\item Generalizing the previous two examples, let $A_\hbar, A_0$, and
  $Z$ be as in Example \ref{a0z0exam}, and $M$ an $A$-bimodule which
  is topologically free over $\KK[[\hbar]]$ and almost symmetric of degree $d$, in
  the sense that $g(z)m \equiv mg(z) \pmod{\hbar^d}$ for all $m \in M$ and $z \in Z$.  
  Then, the Hochschild homology and cohomology of $A_0$
  with coefficients in $M_0$ are Poisson modules over $Z$. Indeed,
  both $A_0$ and $M_0$ are $Z$-modules, and to each $z \in Z$
  one can attach a simultaneous derivation $D_z$ of $A_0$ and
  $M_0$ which is defined up to adding inner derivations, by $D_z(a_0) =
  \hbar^{-d}[g(z), a] \pmod{\hbar}, D_z(m_0) = \hbar^{-d}[g(z), m] \pmod{\hbar}$ for any
  lifts $a \in a_0 + (\hbar)$, and $m \in m_0
  +(\hbar)$ of $a_0$ and $m_0$ to $A$ and $M$ modulo
  $\hbar$. Here, being simultaneous means that 
$$
D_z(a_0 m_0) =
  D_z(a_0) m_0 + a_0 D_z(m_0),\ 
D_z(m_0 a_0) =
  D_z(m_0) a_0 + m_0 D_z(a_0).
$$
Simultaneous derivations act on Hochschild (co)homology of $A_0$ with
coefficients in $M_0$, and inner derivations act trivially.  Hence,
the map $z \mapsto D_z$ yields a well-defined derivation on this
Hochschild (co)homology (as a module over $Z$), which gives a Lie
algebra action of $Z$ on the cohomology.  Moreover,
$D_{zw}=wD_{z} + z D_{w}$, which implies that the action
induces a Poisson module structure over $Z$.
\end{enumerate}
\end{example}
 
\begin{definition} Let $M$ be a Poisson module over a Poisson
algebra $B$. The zeroth Poisson homology $HP_0(B,M)$ 
is the zeroth Lie algebra homology, i.e., the space of coinvariants 
$M/\lbrace{B,M\rbrace}$.
\end{definition}  

\begin{proposition} \label{ssprop} Let $A$ be a nonnegatively filtered algebra 
with ${\rm gr}(A)=A_0$, and let $Z$ be the center of $A_0$. 
Then the space ${\rm gr}(A/[A,A])$ is a quotient of
$HP_0(Z,A_0/[A_0,A_0])$. 
\end{proposition}

Note that, in the case when $A_0$ is commutative, Proposition \ref{ssprop} says
that $\gr HH_0(A)$ is a quotient of $HP_0(A_0)$. This follows from
the Brylinski spectral sequence.

\begin{remark}\label{ssproprem}
  A similar proposition holds, with the same proof, in the more
  general setting of a flat deformation $A_\hbar$ of $A_0$ over
  $\KK[[\hbar]]$.
\end{remark}

\begin{proof}
  It suffices to show that $\gr([A,A]) \supseteq \{Z, A_0\} + [A_0,
  A_0]$. This follows since, for homogeneous $a_0,b_0 \in A_0$, if $[a_0,b_0] \neq 0$ then
  $\gr[a, b] = [a_0,b_0]$ for any lifts $a, b$
  of $a_0, b_0$ to $A$; similarly, if $\{z, a_0\} \neq 0$ for
homogeneous $z \in Z, a_0
  \in A_0$, then $\gr [g(z), a] = \{z, a_0\}$ for any lift
  $a$ of $a_0$ to $A$.
\end{proof}

\begin{remark} Alternatively, the result follows using the
  spectral sequence for the Hochschild homology of $A$. The zeroth
  homology of the first page is $A_0/[A_0, A_0]$. The first nonzero
  differential, ${\bold d}={\bold d}_d: HH_1(A_0) \rightarrow
  A_0/[A_0, A_0]$, sends the cocycle $z \otimes a \in Z^1(A_0)$
  for $z \in Z, a \in A_0$ to the element $\{z, a\}$, and hence the
  image of ${\bold d}$ contains $\{Z, A_0/[A_0, A_0]\} \subseteq A_0/[A_0,
  A_0]$.  In the case that $A_0 = Z$, this is just the Brylinski
  spectral sequence.
\end{remark} 

For any Poisson variety $X$, we will abbreviate ``quasicoherent sheaf
of Poisson modules over $\OO_X$'' as ``Poisson module on $X$,'' and
``coherent sheaf of Poisson modules over $\OO_X$'' as ``coherent
Poisson module,'' analogously to the terminology for $D$-modules.
\begin{definition} 
Let $X$ be a Poisson variety, and 
$N$ be a 
Poisson module on $X$. Then the right $D$-module 
$M(X,N)$ on $X$ is defined by the formula 
$$
M(X,N)=\langle {\rm HVect}(X)\rangle \setminus N\otimes_{\OO_X}D_X,
$$
where the action of vector fields is diagonal.

More generally, if $Y$ is another variety and $\phi: X \rightarrow Y$
is a map, one may define
$$
M_\phi(X, N) = \langle {\rm HVect}(X,\phi^*(\OO_Y)) \rangle \setminus N \otimes_{\OO_X} D_X.
$$
\end{definition}

\begin{remark}
  Let $i: X\to V$ be an embedding of $X$ into a smooth variety.  Then
  $i_*M(X,N)$ can be defined explicitly as the quotient of
  $N\otimes_{\OO_V} D_V$ by the left action of vector fields on $V$
  that preserve the ideal of $X$ and restrict to Hamiltonian vector
  fields on $X$, and by the left action of functions that vanish on
  $X$.  Similarly, $i_* M_\phi(X,N)$ is the quotient of $N
  \otimes_{\OO_V} D_V$ by the left action of vector fields on $V$ that
  preserve the ideal of $X$ and restrict to Hamiltonian vector fields
  on $X$ associated to functions pulled back from $Y$, and by the left
  action of functions that vanish on $X$.
\end{remark}

\begin{example}
$M_\phi(X,\OO_X)=M_\phi(X)$. 
\end{example}

The following proposition is immediate from (and motivates) the definitions.
\begin{proposition}\label{hp0modprop}
  Let $X$ be affine, and $p$ be the map of $X$ to a point. Then the
  underived direct image $p_0(M(X,N))$ is naturally isomorphic to
  $HP_0(\OO_X,N)$. Similarly, given a Poisson map $\phi: X
\rightarrow Y$ of affine Poisson varieties,
  $p_0(M_\phi(X, N)) \cong HP_0(\OO_Y, N)$.  In particular, if
  $M_\phi(X,N)$ is holonomic, then $HP_0(\OO_Y, N)$ is finite-dimensional.
\end{proposition}
Similarly, the other results of this section all extend to the setting of Poisson modules:
\begin{proposition}
\begin{enumerate}
\item[(i)] If $G$ is a finite group acting on $X$ by Poisson automorphisms, $\pi: X \rightarrow X/G$ the quotient map, $\phi: X \rightarrow Y$ a morphism factoring through $\pi$, and $N$ a $G$-equivariant Poisson module, then $\pi^G_*(M_\phi(X,N)) = M_\phi(X/G,\pi_*(N)^G)$.
\item[(ii)] Given an open subset $U \subseteq X$, $M_\phi(U,N)$ is the restriction of $M_\phi(X, N)$ to $U$.
\item[(iii)] If $i: Z \to X$ is a closed Poisson embedding, then there is
a natural epimorphism $\theta_i: M_\phi(X, N) \to i_* M_\phi(Z,i^* N)$.
\end{enumerate}
\end{proposition}
\begin{remark}
  Analogously to Remark \ref{pdr}, one may define the Poisson-de Rham
  homology of $X$ with coefficients in a Poisson module as
  $HP^{DR}_j(X, N) := L_j p_0 M(X,N)$, i.e., the $j$-th derived
  functor of $HP_0(X, -)$ on the category of Poisson modules, applied
  to $N$.  Again, when $X$ is symplectic, this coincides with the de
  Rham cohomology of $X$ with coefficients in the $D$-module $N$
  (cf.~Example \ref{pmexams}.\ref{pmdmex}).
\end{remark}

\section{The holonomicity theorem and applications} 

\subsection{The holonomicity theorem}

\begin{theorem}\label{holo}
  Assume that $X$ has finitely many symplectic leaves.  Then, for any
  finite morphism $\phi: X\to Y$ and any coherent Poisson module $N$
   on $X$, the $D$-module $M_\phi(X,N)$ is holonomic.
In particular, the $D$-module $M_\phi(X)$ is holonomic for any finite 
morphism $\phi$.  
\end{theorem}

\begin{remark}
  Curiously, \cite{Kalss} calls varieties with finitely many
  symplectic leaves ``holonomic''.  This terminology fits perfectly
  with Theorem \ref{holo}.
\end{remark}

\begin{proof} 
  The question is local, so we may assume that $X$ and $Y$ are affine.
  Denote the symplectic leaves of $X$ by $X_j$, for $j=1,...,m$.  Let
  $X_j^r$ be the set of points $x\in X_j$ such that the rank of the
  linear map $d\phi(x)|_{T_xX_j}$ equals $r$.

  Fix an embedding $i: X\to V$ of $X$ into a smooth affine variety
  $V$. It suffices to show that the $D$-module $i_*M_\phi(X,N)$ is
  holonomic. To this end, we will show that the singular support of
  $i_*M_\phi(X,N)$ is Lagrangian in $T^*V$.

  Let $Z\subset T^*V$ be the variety of pairs $(x,p)$ such that
  $x\in X$ and $b(p,d\phi^*(f)(x))=0$ for every $f\in \OO_Y$ (where $b$ denotes 
the Poisson bivector of $X$). Then, it
  is easy to see (by taking the associated graded of the equations
  defining $i_*M_\phi(X,N)$) that the singular support of $i_*M_\phi(X,N)$
  is contained in $Z$. So, it suffices to show that all components of
  $Z$ have dimension at most $d=\dim(V)$.

  Let $\widetilde X^r_j$ be the preimage of $X^r_j$ in $Z$.  Then,
  $\widetilde X^r_j$ is a vector bundle over $X^r_j$ of rank
  $d-r$. The variety $Z$ is the union of the $\widetilde X^r_j$, so it
  suffices to show that $\dim X^r_j\le r$ for all $r$ and $j$.

  It suffices to restrict our attention to a single symplectic leaf,
  i.e., to assume that $X$ is symplectic. As before, denote by $X^r$
  the set of points of $X$ where $d\phi$ has rank $r$.

  Suppose that, for some $r$, the dimension of $X^r$ is $s>r$.  Let
  $X'$ be the closure of $X^r$, and set $Y':=\phi(X')$.  Then,
  $\phi: X'\to Y'$ is a finite map of $s$-dimensional varieties.
  Thus, if $y\in Y'$ is a generic point, for every $x\in
  \phi^{-1}(y)$,
  $d\phi(x)|_{T_xX'}$ is an isomorphism $T_xX'\to T_yY'$.  So, the
  rank of $d\phi(x)$ is at least $s$, which is a contradiction, since
  it can be at most $r$.
\end{proof}

\subsection{Proofs of Theorems \ref{th1} and \ref{th2} and
Corollaries \ref{co1} and \ref{co2}}

Theorem \ref{holo} together with Corollary \ref{fidi} and Proposition \ref{hp0modprop} implies Theorem
\ref{th1}. Corollary \ref{co1} follows immediately, as does Corollary
\ref{co2}, once we set $Y=X/G$, and let $\phi: X\to X/G$ be the
projection.

To prove Theorem \ref{th2},
note that Theorem \ref{th1} implies that
$HP_0(Z,A_0/[A_0, A_0])$ is finite-dimensional. Then,
$A/[A,A]$ is finite-dimensional by Proposition \ref{ssprop}.
We deduce that $A$ has finitely many irreducible finite-dimensional
representations because the characters of distinct irreducible
finite-dimensional representations of $A$ are linearly independent
linear functionals on $A/[A,A]$.

\subsection{Invariant distributions supported at a point}\label{invdis}

Let $X$ be an algebraic variety, and $x\in X$ be a point.  Define a
{\it distribution on $X$ supported at $x$} to be a linear functional
$\xi: \OO_{X,x}\to \mathbb K$ on the local ring of $X$ at $x$ which
annihilates a power of the maximal ideal ${\mathfrak m}_x$.

Now, let $X$ be a Poisson variety, $Y$ be another variety, and
$\phi: X\to Y$ be a morphism.  Let ${\mathcal M}_\phi(X,x)$ be the
space of distributions on $X$ supported at $x$ which are invariant
under Hamiltonian vector fields associated to functions pulled back
from $Y$. In the special case $X=Y$, $\phi=\id$, we denote
${\mathcal M}_\phi(X,x)$ by ${\mathcal M}(X,x)$. The following
proposition follows directly from definitions.

\begin{proposition}\label{invdi1}
  There is a natural isomorphism
  ${\mathcal M}_\phi(X,x)\cong \Hom(M_\phi(X),\delta_x)$, where
  $\delta_x$ is the delta-function $D$-module of $x \in X$.
  In particular, ${\mathcal M}(X,x)\cong \Hom(M(X),\delta_x)$.
\end{proposition}  

\begin{remark} 
  It is easy to see that ${\mathcal M}(X,x)\ne 0$ if and only if
  $x\in X$ is a one-point symplectic leaf. More generally, if $x$ is a
  point of an arbitrary (locally closed) symplectic leaf
  $X_0 \subseteq X$, then $\mathcal{M}_\phi(X,x) \neq 0$ if and only
  if the restriction of $d \phi$ to $T_x X_0$ is zero.  Note also
  that, when $X$ and $Y$ are affine, ${\mathcal M}_\phi(X,x)$ is a
  subspace of the space $(\OO_X/\lbrace{\OO_Y,\OO_X\rbrace})^*$.
\end{remark}

Now suppose that $X$ has finitely many symplectic leaves, and that
$\phi$ is a finite morphism. In this case, Theorem \ref{th1} implies
that ${\mathcal M}_\phi(X,x)$ is finite-dimensional. We now obtain an
explicit upper bound for the dimension of this space.

We may assume that $X$ and $Y$ are affine. Fix a closed embedding $i$
of $X$ into a finite-dimensional vector space $V$, so that $x$ maps to
the origin. Let $\widetilde {Z}$ be the closed subscheme of $T^*V$
defined by the equation $b(p,d\phi^*(f)(v))=0$, $v\in i(X)$ (its
reduced part is the variety $Z$ from the proof of Theorem \ref{holo}).
Let $f_1, \ldots, f_m$ be generators of $\OO_Y$, and $w_j$, $j=1, \ldots, m$, be
the Hamiltonian vector fields on $X$ associated to the functions
$\phi^*(f_j)$. Let $v_j$ be liftings of $w_j$ to vector fields on $V$,
and let $g_k$, $k=1, \ldots, r$, be generators of the defining ideal $I_X$
of $X$ in $V$. For every $p\in V^*$, define the ideal
$J_p' \subset \OO_V$ generated by the elements $g_k$ and $(v_j,p)$.  It
follows from the proof of Theorem \ref{holo} that, for Zariski generic
$p$, the ideal $J_p'$ has finite codimension.  Furthermore, let $J_p \supseteq J_p'$ be the primary component of $J_p'$ corresponding to the point $(0,p)$. 
Note that, for generic $p$, the
 codimension of $J_p$
is the multiplicity of the component $V^*$ of the scheme
$\widetilde{Z}$ (which is well-defined since $Z$ is Lagrangian).
\begin{proposition}\label{invdi2} 
  If $X$ has finitely many symplectic leaves and $\phi$ is finite,
  then the dimension of ${\mathcal M}_\phi(X,x)$ is dominated by the
  codimension of the ideal $J_p$ for all $p \in V^*$.
\end{proposition} 
\begin{proof}
  This follows from Proposition \ref{invdi1}, since the dimension of
  $\Hom(M_\phi(X),\delta_x)$ is dominated by the multiplicity of
  $V^*$ in the singular support of $M_\phi(X,i)$, which is the
  codimension of $J_p$ for generic $p$, and hence the minimal
  codimension.
\end{proof}
As remarked above, this upper bound on the dimension of
$\Hom(M_\phi(X),\delta_x)$ is the multiplicity of the component of
$V^*$ in the scheme $\widetilde{Z}$.
\begin{remark}\label{invdi2.5} 
  Note that the proof of Proposition \ref{invdi2} has a purely local
  nature. It does not involve direct images to the point, and uses
  only basic properties of holonomic $D$-modules.  In particular, it
  also applies to analytic varieties, for which, as we mentioned,
  $HP_0$ may be infinite-dimensional (for a reference on the theory of
  $D$-modules on analytic varieties, see, e.g., \cite{Bj}).
\end{remark}

\begin{proposition}\label{invdi3} 
  Suppose that $X$ and $Y$ are affine, and that $\OO_X$ and $\OO_Y$
  are nonnegatively graded with finite-dimensional graded components.
  Assume also that the Poisson bracket on $X$ is homogeneous, and that
  the map $\phi^*$ preserves degree.  Then, under the assumptions of
  Proposition \ref{invdi2},
  ${\mathcal M}_\phi(X,x)=(\OO_X/\lbrace{\OO_Y,\OO_X\rbrace})^*$.
  Hence, $\dim(\OO_X/\lbrace{\OO_Y,\OO_X\rbrace})\le {\rm
    codim}(J_p)$.
\end{proposition}
\begin{remark}
  In the case that the Poisson bracket has negative degree, the
  assumption that each graded component is finite-dimensional is
  superfluous.  Indeed, it is enough to show that the degree-zero part
  is finite-dimensional. Since it is Poisson central, if it were not
  finite-dimensional, then the Poisson center itself would be
  infinite-dimensional, and in particular admit infinitely many
  distinct characters.  However, for every character of the Poisson
  center, the associated subvariety is Poisson and hence contains at
  least one symplectic leaf, so there can only be finitely many.
\end{remark}
\begin{remark}
  In the graded situation of the proposition, $J_p = J_p'$, since its
  support, projected to $X$, is conical and zero-dimensional.
\end{remark}
\begin{proof}
  The first statement follows from the fact that the space
  $\OO_X/\lbrace{\OO_Y,\OO_X\rbrace}$ is nonnegatively graded and
  finite-dimensional.  The second statement follows from the first one
  and Proposition \ref{invdi2}.
\end{proof}

Proposition \ref{invdi3} combined with computer algebra systems can be
used to obtain explicit upper bounds for the dimension of
$\OO_X/\lbrace{\OO_Y,\OO_X\rbrace}$ in specific examples, e.g., when
$X=V$ is a finite-dimensional symplectic vector space and $Y = V/G$,
where $G\subset Sp(V)$. This will be discussed in more detail in a
subsequent paper.

\subsection{Generalization to Poisson schemes of finite type}\label{sche}

It is straightforward to generalize the above results to the case
where $X$ and $Y$ are separated Poisson schemes of finite type, which
may be non-reduced.  Indeed, if $X$ is a Poisson scheme, $Y$ is
another scheme (both separated and of finite type over $\mathbb K$), and
$\phi: X \rightarrow Y$ is a morphism, then we may define $M_\phi(X)$,
and hence also $M(X)$, in the same way as when $X$ and $Y$ were
varieties.\footnote{Recall that a $D$-module on $X$ is, by definition,
  the same as a $D$-module on the reduced part $X_{\mathrm{red}}$.}
The results of \S 2 continue to hold without change.

Recall from \cite[Corollary 1.4]{Kalnpa} that, if $X$ is a Poisson
scheme, its reduction, $X_{\mathrm{red}}$, is also Poisson.  We say that $X$
has finitely many symplectic leaves if so does $X_{\mathrm{red}}$.  Using this
terminology, we claim that Theorem \ref{holo}, and hence all the
results of \S 1, extend to the case of Poisson schemes of finite type.

Indeed, it suffices to show that, if $X$ is affine, and $i: X\to V$ is
a closed embedding of $X$ into a smooth affine variety $V$, then the
singular support of $i_*M_\phi(X,N)$ is Lagrangian in $T^*V$.  This
follows as in the reduced case, replacing $X$ and $Y$
with $X_{\mathrm{red}}$ and $Y_{\mathrm{red}}$.

\subsection{The case of formal schemes}\label{fsche}

We expect that the above results can be generalized to the setting of
formal schemes of finite type, at least under the assumption that the
schemes in question are affine. In this paper, we will give only two
results in this direction, in order to prove Corollary \ref{fintrace}
and its generalization, Corollary \ref{fintracepm}, which are used in
the appendix.

Let $X$ be an affine formal scheme of finite type over ${\mathbb K}$
with a unique closed point $x$.  This means that $A=\OO_X$ is a local
${\mathbb K}$-algebra which is separated, complete, and topologically
finitely generated in the adic topology defined by its maximal ideal
${\mathfrak m}_x\subset A$.  Assume that $X$ is a Poisson scheme
(i.e., $A$ is a Poisson algebra), and that the Poisson bracket
vanishes at $x$ (i.e., ${\mathfrak m}_x$ is a Poisson ideal). In this
case, the finite-dimensional algebras $A_n:=A/{\mathfrak m}_x^n$ are
Poisson, and the maps $A_{n+1} \to A_n$ induce surjections
$HP_0(A_{n+1})\to HP_0(A_n)$.

\begin{proposition}\label{clo} 
  \begin{enumerate}
  \item[(i)] The subspace $\lbrace{A,A\rbrace}$ is closed in $A$.
  \item[(ii)] $HP_0(A)=\underleftarrow{\lim}\, HP_0(A_n)$.
  \end{enumerate}
\end{proposition}

\begin{proof} 
  Part (ii) follows immediately from (i), since
  $\underleftarrow{\lim}\, HP_0(A_n) = A / \overline{\{A,A\}}$.  To
  prove part (i), suppose that $A$ is topologically generated by
  $x_1, \ldots, x_n\in \mathfrak m_x$.  Then,
  $\lbrace{A,A\rbrace}=\sum_{i=1}^n \lbrace{x_i,A\rbrace}$.
  Therefore, (i) follows from the following (well known) lemma, by
  setting $V := A^n$, $W := A$, and
  $L(g_1, \ldots, g_n)=\sum_i \lbrace{x_i,g_i\rbrace}$.
\end{proof}
\begin{lemma}
  Let $V,W$ be linearly compact topological vector spaces (i.e.,
  projective limits of finite-dimensional vector spaces), and let
  $L: V\to W$ be a continuous linear operator. Then, the image of $L$
  is closed.
\end{lemma}
\begin{proof}
  The continuous duals $V^*,W^*$ are ordinary vector spaces (possibly
  infinite-dimensional), and $\mathrm{Im}(L)=(\mathrm{Ker} L^*)^\perp$.
\end{proof}

\begin{proposition}\label{finit}
  Let $A$ and $X$ be as above, and assume that $X$ has finitely many
  symplectic leaves.  Then, the space $HP_0(A)$ is finite-dimensional.
\end{proposition}

\begin{proof}
  Let $HP_0(A)^*$ be the space of continuous linear functionals on
  $HP_0(A)$ (in the topology of inverse limit). Then
  $HP_0(A)^*={\mathcal M}(X,x)$, where ${\mathcal M}(X,x)$ is defined
  in the same way as in the case when $X$ is a variety.  Thus, by the
  formal analogue of Proposition \ref{invdi2} (which is proved
  similarly to the case of usual varieties) we deduce that $HP_0(A)^*$
  is finite-dimensional.  Hence, so is $HP_0(A)$.
\end{proof}

\begin{corollary}\label{fintrace}
Let $A_0={\mathcal O}_X$ be a local $\mathbb{K}$-algebra 
which is separated and complete in the adic topology. 
Let $A$ be a flat formal deformation of $A_0$ over $\mathbb{K}[[\hbar]]$,
such that $A/\hbar A = A_0$. If $\Spec A_0$ has finitely many symplectic leaves with the Poisson structure induced from the commutator on $A$, then
$(A/[A,A])[\hbar^{-1}]$ is a finite-dimensional vector space over $\mathbb{K}((\hbar))$.
\end{corollary}
\begin{proof}
  There is a canonical inclusion $((A/[A,A]) [\hbar^{-1}])^* \subseteq
  (HP_0(A_0)((\hbar)))^*$, taking duals over $\KK((\hbar))$, which
  follows as in the proof of Proposition \ref{ssprop}; see also Remark
  \ref{ssproprem}.  By Proposition \ref{finit}, the latter is
  finite-dimensional.
\end{proof}
%  More directly, if $\phi([a,b]) = 0$ for all $a \in A_{\leq i}, b \in A_{\leq j}$, and $\phi: A \rightarrow \C$
% kills elements of degree $\leq (i+j-2)$, then $(\gr_{i+j-1} \phi)(\{\gr_i a, \gr_j b\}) = 0$ for all $a \in A_{\leq i}, b \in A_{\leq j}$.

\subsection{Poisson modules for formal schemes}
The results of \S \ref{invdis}--\ref{fsche} generalize without
significant changes to the setting of Poisson modules. Since the
proofs are generalized in a straightforward manner, we will give the
statements only.  Our main goal is Corollary \ref{fintracepm}, which
is used in the appendix.

We keep the notation of \S 3.3.  Let $N$ be a coherent Poisson module
on $X \subseteq V$.  Define {\it a distribution on
  $N$ supported at $x\in X$} to be a linear functional $\xi: {\mathcal
  O}_{X,x} \otimes_{{\mathcal O}_X}N\to \Bbb K$ which vanishes on
${\mathfrak m}_x^n N$ for large enough $n$.

Let ${\mathcal M}_\phi(X,N,x)$ be the space of distributions on $N$
supported at $x$ which are invariant under Hamiltonian vector fields 
defined by functions $\phi^*f$, $f\in {\mathcal O}_Y$. 
In particular, if $X=Y$ and $\phi={\rm id}$, this space will be denoted 
by ${\mathcal M}(X,N,x)$.  

The following proposition is a direct generalization of Proposition
\ref{invdi1}.

\begin{proposition}\label{invdi1pm}
  There is a natural isomorphism
  ${\mathcal M}_\phi(X,N,x)\cong \Hom(M_\phi(X,N),\delta_x)$.
%  In particular, ${\mathcal M}(X,N,x)\cong \Hom(M(X,N),\delta_x)$.
\end{proposition} 

Now let us assume that $X$ has finitely many symplectic leaves and
$\phi$ is finite, and give an explicit upper bound for the dimension
of ${\mathcal M}_\phi(X,N,x)$ (in particular, showing that it is
finite-dimensional). For $p \in V^*$, let $N_p\subset N$ be the
${\mathcal O}_X$-submodule defined by
$$
N_p= \sum_j (v_j,p)|_X N
$$
(so, if $N={\mathcal O}_X$ then $N_p$ is the image of $J_p$ in ${\mathcal O}_X$). 
Then $N_p$ has finite codimension in $N$ for generic $p$. Namely, this codimension is 
the generic rank on $V^*\subset X\times V^*$ 
of the module $N\otimes_{{\mathcal O}_X} {\mathcal O}(V\oplus V^*)/\langle (v_j,p)\rangle $, 
which is defined by the classical limits of the equations defining $i_*{\mathcal M}_\phi(X,N,x)$. 

We have the following direct generalizations of Propositions \ref{invdi2}
and \ref{invdi3}:
\begin{proposition}\label{invdi2pm} 
  If $X$ has finitely many symplectic leaves and $\phi$ is finite,
  then the dimension of ${\mathcal M}_\phi(X,N,x)$ is dominated by the
  codimension of the submodule $N_p$ for all $p \in V^*$.
\end{proposition} 

\begin{proposition}\label{invdi3pm} 
  Suppose that $X$ and $Y$ are affine, and that ${\mathcal O}_X$, ${\mathcal O}_Y$, and $N$
  are nonnegatively graded with finite-dimensional graded components.
  Assume also that the Poisson bracket on $X$ and the action of ${\mathcal O_X}$ on $N$ 
  are homogeneous, and that
  the map $\phi^*$ preserves degree.  Then, under the assumptions of
  Proposition \ref{invdi2pm},
  ${\mathcal M}_\phi(X,N,x)=(N/\lbrace{{\mathcal O}_Y,N\rbrace})^*$.
  Hence, $\dim(N/\lbrace{{\mathcal O}_Y,N\rbrace})\le {\rm
    codim}(N_p)$.
\end{proposition} 

Now let us consider the setting of formal schemes with one closed
point, as in Subsection 3.5.  We keep the notation of Subsection
3.5. Let $N$ be a coherent Poisson $A$-module, and let
$N_n:=N/{\mathfrak m}_x^nN$.  Then we have surjective maps
$HP_0(A_{n+1},N_{n+1})\to HP_0(A_n,N_n)$.

Proposition \ref{clo} generalizes to
\begin{proposition}\label{clopm} 
  \begin{enumerate}
  \item[(i)] The subspace $\lbrace{A,N\rbrace}$ is closed in $N$.
  \item[(ii)] $HP_0(A,N)=\underleftarrow{\lim}\, HP_0(A_n,N_n)$.
  \end{enumerate}
\end{proposition}

This allows us to obtain the following generalization of Proposition
\ref{finit} and Corollary \ref{fintrace}.

\begin{proposition}\label{finitpm}
  Assume that $A$, $X$, $N$ are as above, and $X$ has finitely many
  symplectic leaves.  Then, the space $HP_0(A,N)$ is finite-dimensional.
\end{proposition}

\begin{proof}
  Let $HP_0(A,N)^*$ be the space of continuous linear functionals on
  $HP_0(A,N)$ (in the topology of inverse limit, defined thanks to Proposition \ref{clopm}). Then
  $HP_0(A,N)^*={\mathcal M}(X,N,x)$, where ${\mathcal M}(X,N,x)$ is defined
  just as in the case $X$ is a variety.  By the
  formal analogue of Proposition \ref{invdi2pm},  $HP_0(A,N)^*$
  is finite-dimensional.  Hence, so is $HP_0(A,N)$.
\end{proof}

\begin{corollary}\label{fintracepm}
Let $A_0={\mathcal O}_X$ be a local $\Bbb K$-algebra 
which is separated and complete in the adic topology. 
Let $B_0$ be an algebra over $A_0$ which is finitely generated 
as an $A_0$-module, and whose center is $A_0$. Let $B$ 
be a flat formal defomation of $B_0$ over $\Bbb K[[\hbar]]$, such that   
$B/\hbar B=B_0$, and assume that the Poisson bracket on 
$X$ induced by this deformation corresponding to some $d>0$ (as in Remark \ref{hbrem}) has finitely many symplectic leaves. 
Then $(B/[B,B])[\hbar^{-1}]$ is a finite-dimensional vector space over $\Bbb K((\hbar))$. 
\end{corollary}

\begin{proof}
  By the proof of Proposition \ref{ssprop} (cf.~Remark \ref{ssproprem}),
  $\dim_{\Bbb K((\hbar))}(B/[B,B])[\hbar^{-1}]$ is dominated by $\dim
  HP_0(A_0,B_0/[B_0,B_0])$. Since $B_0/[B_0,B_0]$ is a coherent Poisson
  module over $A_0$,  Proposition \ref{finitpm} implies that
  the space $HP_0(A_0,B_0/[B_0,B_0])$ is finite-dimensional.
\end{proof} 

\section{The structure of $M(X)$ when $X$ has finitely many symplectic leaves}

\subsection{The general structure of $M(X)$}

For a variety $X$, let $IC(X)$ denote the intersection cohomology
(right) $D$-module on $X$, which is the intermediate extension
$j_{!*}(\Omega_U)$, for any open embedding $j: U \hookrightarrow X$ of
a smooth dense subvariety $U$.

Let $X$ be a Poisson variety.

\begin{proposition}\label{locsys}
  If $X$ has finitely many symplectic leaves, then $M(X)$ has a finite
  Jordan-H\"older series, in which all composition factors are
  intermediate extensions of irreducible local systems (i.e.,
  $\OO$-coherent right $D$-modules) on symplectic leaves of
  $X$. Moreover, the only such local system supported on each open
  symplectic leaf is the trivial local system $\Omega$, with
  multiplicity one. Hence, the Jordan-H\"older series of $M(X)$
  contains $IC(X)$ with multiplicity one.
\end{proposition}

\begin{proof} 
  Let $i: X\to V$ be a closed embedding of $X$ into a smooth variety.
It follows from the proof of Theorem \ref{holo} that the singular support of
  $M_{\id}(X,i)$ is the union of the conormal bundles in $T^*V$ of the
  symplectic leaves $X_j$ of $X$. Now recall that if $W$ is any smooth algebraic variety with a finite 
stratification, and $M$ is a holonomic $D$-module on $W$ whose singular support is contained 
in the union of the conormal bundles of the strata, then all members of the composition series of $M$ 
are irreducible extensions of irreducible local systems from the strata (see \cite{Kash}). 
This implies the first statement of the Proposition.  

The second statement follows from Lemma
  \ref{locpm}. Indeed, if $U\subset X$ is 
an open symplectic leaf, then by Lemma 
\ref{locpm} $M(X)|_U=M(U)=\Omega_U$. 
\end{proof}

\begin{remark} 
  The proposition extends to the case of $M_\phi(X)$, where
  $\phi: X \rightarrow Y$ is finite and $X$ has finitely many
  symplectic leaves, provided we replace the symplectic leaves in the
  statement of the proposition by the loci in the symplectic leaves of
  constant rank of $d \phi$, which we denoted by $X_j^r$ in the proof
  of Theorem \ref{holo}.
\end{remark}

\subsection{An example of a computation of $M(X)$}

Let $X$ be a Kleinian surface of type ADE (with its standard Poisson
structure), and $j: X\setminus 0\to X$ be the corresponding open
embedding.  In this case, the composition factors of $M(X)$ are the
intermediate extension $j_{!*}(\Omega)$ of $\Omega$, i.e., the
intersection cohomology $D$-module, $IC(X)$, which occurs with
multiplicity 1, and the $\delta$-function $D$-module, $\delta$, at the
origin.

\begin{lemma}\label{semis}
  Let $X$ be an irreducible affine variety of dimension $d$ with a
  ${\mathbb K}^\times$-action having a unique fixed point $0$, which is
  attracting (i.e., $X$ is a cone). Let $U=X\setminus 0$.  Assume that $U$ is
  smooth. Let $j: U\to X$ be the corresponding open embedding.  Then,
  for $m>0$,
  \[ 
  \Ext^m(j_{!*}(\Omega),\delta)=  \Ext^m(\delta,j_{!*}(\Omega))=H^{d-m}(U,{\mathbb K})^*. 
  \]
\end{lemma}

\begin{proof}
  The first equality follows easily from Verdier duality, since the $D$-modules 
$\delta$ and $j_{!*}(\Omega)$ are self-dual.  We prove
  the second equality.  Let $i: 0 \hookrightarrow X$ be the
  inclusion. Using the adjunction of $i^!$ and $i_*$, 
we have
 \[ 
  \Ext^m(\delta,j_{!*}(\Omega))= 
  \Ext^m(i_*\Bbb K,j_{!*}(\Omega))= 
  \Ext^m(\Bbb K,i^!j_{!*}(\Omega)). 
  \]
Now, since $i^!$ and $i^*$ are interchanged by Verdier duality,
we have 
\[
  \Ext^m(\Bbb K,i^!j_{!*}(\Omega))=
  \Ext^m(\Bbb D i^!j_{!*}(\Omega),\Bbb K)= 
\]
\[
  \Ext^m(i^*\Bbb D j_{!*}(\Omega),\Bbb K)= 
  \Ext^m(i^*j_{!*}(\Omega),\Bbb K)=
H^{-m}(i^*j_{!*}(\Omega))^*.
  \]
  There is a standard exact triangle
  \[ \to j_{!*}(\Omega)\to j_*(\Omega) \to K\to, \]
  where $K$ is a complex concentrated at $0$ and in nonnegative degrees.
  Applying $i^*$ to this triangle, we obtain
  \[
  \to i^*j_{!*}(\Omega)\to i^*j_*(\Omega) \to i^*K\to,
  \]
  where $i^*K$ is concentrated in nonnegative degrees.  Since
  $i^*j_{!*}(\Omega)$ is concentrated in negative degrees,
  $i^*j_{!*}(\Omega)=\tau^{<0}(i^*j_*(\Omega))$, where $\tau^{<0}$ is the truncation. 
  Thus, $H^{-m}(i^*j_{!*}(\Omega))=H^{-m}(i^*j_*(\Omega))$. Since $X$ is
  conical, this is naturally isomorphic to $H^{d-m}(X, j_*(\Omega))$,
  and hence to $H^{d-m}(U, \Omega) \cong H^{d-m}(U,{\mathbb K})$.
\end{proof}

\begin{corollary}\label{semis1}
If $X$ is a Kleinian surface, then 
\[
\Ext^1(\delta,j_{!*}(\Omega))=
\Ext^1(j_{!*}(\Omega),\delta)=0.
\]
\end{corollary}

\begin{proof}
In the Kleinian case, $d=2$ and $H^1(U)=0$, since 
$U=(\mathbb A^2\setminus 0)/\Gamma$, where 
$\Gamma$ is a finite subgroup in $SL_2$.
Thus, the corollary is a special case of Lemma \ref{semis}.  
\end{proof} 

\begin{corollary}\label{klein}
  If $X$ is a Kleinian surface, then
  $M(X)=j_{!*}(\Omega)\oplus \mu \delta$, where $\mu$ is the Milnor
  number of $X$.
\end{corollary}

\begin{proof}
  By Corollary \ref{semis1}, $M(X)=j_{!*}(\Omega)\oplus n \delta$, for
  some nonnegative integer $n$.  Also, it is known that the
  intersection cohomology of the Kleinian surface equals its ordinary
  cohomology, since the Kleinian surface is the quotient of an affine
  plane by a finite group. The $i$-th intersection cohomology group is
  $H^{i-\dim X}(p_*(IC(X))) = H^{i-\dim X}(p_*(j_{!*} \Omega))$, where
  $p_*$ is the (derived) direct image under the projection
  $p: X \to pt$.  Thus, the underived direct image
  $p_0(j_{!*}(\Omega))$ is zero, and hence $\dim p_0(M(X))=n$.  By
  Proposition \ref{hp0}, $\dim HP_0(\OO_X)=n$.  On the other hand,
  \cite{AL} shows that $\dim HP_0(\OO_X)=\mu$. We deduce that
  $n = \mu$.
\end{proof} 

\subsection{The local systems attached to symplectic leaves}\label{locs}

We now give more precise information about the structure of $M(X)$. 

For each symplectic leaf $S$ of $X$, let $U$ be the complement in $X$
of $\bar S\setminus S$, and $i_S: S\to U$ be the corresponding closed
embedding.  Set $L_S=H^0(i_S^*(M(U)))$. It is easy to see that $L_S$
is a local system. Moreover, there is a surjective adjunction morphism
$\phi_S: M(U)\to (i_S)_*L_S$.  In particular, the intermediate
extensions of all composition factors of $L_S$ are composition factors
of $M(X)$.

To compute the local systems $L_S$, we use the equality
\begin{equation}\label{lseqn}
  \Hom(L_S,N)=\Hom(M(U),(i_S)_*(N))=(i_S)_*(N)^{\OO_U},
\end{equation}
for any local system $N$ on $S$, where the superscript $\OO_U$ denotes
the invariants with respect to Hamiltonian vector fields.

For $s\in S$, let $X_s$ be the formal neighborhood of $s$ in $X$.  By
the formal Darboux-Weinstein theorem (\cite{We}; see also
\cite[Proposition 3.3]{Kalss}), $X_s$ is Poisson isomorphic to the
product $X_s^0\times S_s$, where the ``slice'' $X_s^0$ is a reduced
formal Poisson scheme with a unique closed point $0$ and a Poisson
structure vanishing at this point.

Consider the space $HP_0(\OO_{X_s^0})$.  Since $X_s$ has finitely many
symplectic leaves (which are the intersections of $X_s$ with the
symplectic leaves of $X$), so does $X_s^0$, and thus
$HP_0(\OO_{X_s^0})$ is finite-dimensional by Proposition \ref{finit}.

Let $\delta_{S_s}:=(i_{S_s})_*\Omega_{S_s}$ be the \emph{delta-function $D$-module}
of $S_s\subset X_s$.  Let
$\mathcal M(X_s, S_s) := \Hom(M(X_s), \delta_{S_s})$.

\begin{lemma}\label{cano} 
  The space $HP_0(\OO_{X_s^0})^*$ is canonically isomorphic to
  $\mathcal M(X_s, S_s)$.
\end{lemma} 

\begin{proof} 
  By Proposition \ref{invdi1},
  $M(X_s) \cong M(X_s^0) \boxtimes M(S_s) \cong M(X_s^0) \boxtimes
  \Omega_{S_s}$,
  and
  $\Hom(M_{X_s^0}, \delta_0) = \mathcal M(X_s^0, 0) \cong
  HP_0(\OO_{X_s^0})^*$.
\end{proof}

Hence, the space $HP_0(\OO_{X_s^0})$ is
canonically attached to the point $s$, and is independent (up to a
canonical isomorphism) of the Darboux-Weinstein decomposition of
$X_s$.

\begin{proposition}\label{flatconn}
  The family $s\mapsto HP_0(\OO_{X_s^0})={\mathcal M}(X_s,S_s)^*$ is a
  vector bundle on $S$ which carries a natural flat connection. In
  other words, $HP_0(\OO_{X_s^0})$ is a local system.
\end{proposition}

\begin{remark} 
  To be more precise, a vector bundle with a flat connection is
  naturally a left $D$-module, while we are working in the setting of
  right $D$-modules. So, strictly speaking, we must tensor our flat
  bundles with the canonical sheaf. Luckily, all flat bundles we
  consider are on symplectic varieties, so the canonical sheaf is
  canonically trivial, and the issue does not arise.
\end{remark}   

\begin{proof}
  Any Darboux-Weinstein decomposition
  $\zeta: X_s\cong X_s^0\times S_s$ trivializes the family in question
  in the formal neighborhood of $s$, which shows that this family is
  indeed a vector bundle.  Moreover, this trivialization defines a
  flat connection $\nabla_\zeta$ on this bundle in the formal
  neighborhood of $s$. Since the Lie algebra of Hamiltonian vector
  fields on $X_s$ acts trivially on ${\mathcal M}(X_s,S_s)$, the
  connection $\nabla_\zeta$ does not really depend on the choice of
  $\zeta$, and hence is defined globally on $S$.
\end{proof}

\begin{proposition}\label{rankls}
  There is a natural isomorphism of local systems
  $\psi_S: L_S\cong HP_0(\OO_{X_s^0})$.
\end{proposition} 
\begin{proof}
  Consider first the case when $S=s$ is a single point (and $X=U$).
  In this case, let $\delta := \delta_s$ be the delta-function $D$-module
  of $s \in S$. By \eqref{lseqn},
  \[
  L_S(s)=\Hom(M(X),\delta)^*,
  \]
  But, $\Hom(M(X),\delta)$ is the space of Hamiltonian invariant
  distributions on $X$ supported at $s$, which is canonically
  isomorphic to $HP_0(\OO_{X_s^0})^*$.

  In the general case, for each $s\in S$ fix a Darboux-Weinstein trivialization 
$X_s\cong X_s^0\times S_s$, as explained above. Then the above construction defines an isomorphism 
$\psi_s: L_S(s)\cong HP_0(\OO_{X_s^0})$. It is easy to check that this isomorphism 
does not depend on the choice of the Darboux-Weinstein trivialization (by relating any two such 
trivializations by a formal Hamiltonian automorphism). Therefore, the collection of maps $\psi_s$, $s\in S$,  
defines a canonical isomorphism of vector bundles 
$\psi_S: L_S\cong HP_0(\OO_{X_s^0})$. Moreover, $\psi_S$ 
preserves the flat connections on these bundles, since upon 
a choice of a Darboux-Weinstein trivialization, 
the two bundles and connections become trivial, and 
the map $\psi_s$ becomes independent on $s$.
\end{proof}

\begin{example}
  The following is an example of computation of $L_S$, which also
  demonstrates that the local systems $L_S$ (and hence the local
  systems in Proposition \ref{locsys}) need not have regular
  singularities.

  Let $Z$ be the cone of the elliptic curve $F=x^3+y^3+z^3=0$,
  equipped with the standard Poisson structure, given generically by
  the symplectic form $\frac{dx\wedge dy\wedge dz}{dF}$. This Poisson
  structure has degree zero, so the Euler derivation
  $E: \OO_Z\to \OO_Z$ is Poisson.  Let $X=Z\times \mathbb A^2$, where
  $\mathbb A^2$ has coordinate functions $p$ and $q$ with
  $\lbrace{p,f\rbrace}=Ef$ and $\lbrace{q,f\rbrace}=0$ for
  $f\in \OO_Z$, and $\lbrace{p,q\rbrace}=1$. Then, $X$ is a Poisson
  variety with two symplectic leaves: the open four-dimensional leaf,
  and the two-dimensional leaf $S=0\times \mathbb A^2$.

\begin{proposition}
  $L_S$ is isomorphic to $N_0\oplus 3N_1\oplus 3N_2\oplus N_3$, where
  $N_m$, $m\ge 0$, is the quotient of the algebra of differential
  operators in $p$ and $q$ by the right ideal generated by
  $\partial_q-m$ and $\partial_p$ (i.e., $N_0=\Omega$).
\end{proposition} 

\begin{proof} 
  According to \cite{AL}, the space $HP_0(\OO_Z)$ is naturally
  isomorphic to the Jacobi ring of the elliptic singularity $F=0$
  (i.e., the ring generated by $x,y,z$ with the relations
  $F_x=F_y=F_z=0$).  This ring has a basis
  $(1,x,y,z,xy,xz,yz,xyz)$. This implies that $\dim HP_0(\OO_Z)=8$, and
  the eigenvalues of $E$ on $HP_0(\OO_Z)$ are $0,1,1,1,2,2,2$, and $3$.  So,
  the result follows from Proposition \ref{rankls} and the equalities
  $\lbrace{p,f\rbrace}=Ef$ and $\lbrace{q,f\rbrace}=0$.
\end{proof}
\end{example} 

\subsection{The structure of $M_\phi(X)$ for linear symplectic quotients}

Let $V$ be a finite-dimensional symplectic vector space, $G$ be a
finite subgroup of $Sp(V)$, and $\phi: V\to V/G$ be the tautological
map. We call a subgroup $K\subset G$ \emph{parabolic} if there exists
$v\in V$ such that $K=G_v$, the stabilizer of $v$.  Let $\Par(G,V)$
denote the set of parabolic subgroups.  For a parabolic subgroup $K$,
denote by $H(K)$ the space
$\OO_{(V^{K})^\perp}/\lbrace{\OO_{(V^K)^\perp}^K,\OO_{(V^K)^\perp}\rbrace}$.
Let $i_K: V^K\to V$ be the embedding of $V^K$ into $V$.

\begin{theorem}\label{quotsin}
  There is a canonical $G$-equivariant isomorphism
  \[
  M_\phi(V)\cong \bigoplus_{K\in \Par(G,V)}H(K)\otimes
  (i_K)_*(\Omega_{V^K}).
  \]
\end{theorem}

\begin{proof} 
  Define a stratification of $V$ into strata
  $(V^K)^\circ := \{v \in V \mid \Stab(v) = K\} \subseteq V^K$, for
  $K \in \Par(G,V)$.  The rank of $d\phi$ is constant on
  $(V^K)^\circ$ and equal to $\dim(V^K) = \dim(V^K)^\circ$. Thus, it
  follows from the proof of Theorem \ref{holo} that all composition
  factors of $M_\phi(V)$ are of the form $(i_K)_*(\overline{L})$,
  where $\overline{L}$ is the intermediate extension of a local system
  $L$ on $(V^K)^\circ$.

  We claim that $L$ (and hence $\overline{L}$) is necessarily the
  trivial local system. To see this, it suffices to restrict
  $M_\phi(V)$ to the formal neighborhood of $(V^K)^\circ$, which is of
  the form $(V^K)^\circ \times ((V^K)^\perp)_0$, where
  $((V^K)^\perp)_0$ is the formal neighborhood of the origin in
  $(V^K)^\perp$.  Then, the statement reduces to the case $K=G$, with
  $V^K=0$, which is trivial.

  Moreover, we claim that the multiplicity space for this local system
  is naturally identified with $H(K)$. This reduces in the same way to
  the case $K=G$, where the multiplicity space
  $\Hom(M_\phi(V), \delta_0)^* = \mathcal M_\phi(V,0)^*$ equals
  $\OO_V/\{\OO_{V/G}, \OO_V\} = H(G)$ by Proposition \ref{invdi3}.
 
  Finally, let us show that $\Ext^1$ between any two
  $D$-modules of the form $(i_K)_*(\Omega_{V^K})$ is zero. This will imply
that $M_\phi(V)$ is semisimple, and the desired result will follow.

To this end, note that all the vector spaces $V^K$ are symplectic, i.e., 
  even-dimensional. Hence, the vanishing of $\Ext^1$ is a consequence of the following lemma. 

\begin{lemma}
Let $V_1,V_2\subset V$ be two subspaces of a finite-dimensional vector space $V$, and 
$\delta_{V_i}$ be the right delta-function $D$-modules of $V_i$. Then, if 
$\Ext^1(\delta_{V_1},\delta_{V_2})\ne 0$ then either $V_1\subset V_2$ and $\dim V_2/V_1=1$,
or $V_2\subset V_1$ and $\dim V_1/V_2=1$. 
\end{lemma}

\begin{proof}
First, recall that for right $D$-modules on the line, 
one has 
\begin{equation}\label{vani}
\Ext^1(\Omega,\Omega)=\Ext^1(\delta,\delta)=0.
\end{equation}
Next, pick a basis $B=\lbrace{v_1,...,v_n\rbrace}$ of $V$ compatible with $V_1,V_2$. 
For any $j\in [1,n]$ and $i=1,2$, let $M_i^j$ be the $D$-module 
on the line $\Bbb Kv_j$ defined by the rule: $M_i^j=\Omega$ if $v_j\in V_i$ and $M_i^j=\delta$ if not. 
Then we have
$$
\delta_{V_i}=\boxtimes_{j=1}^n M_i^j.
$$
Using this decomposiion, equality (\ref{vani}), and the K\"unneth formula, 
we see that if $\Ext^1(\delta_{V_1},\delta_{V_2})\ne 0$ then $M_1^j=M_2^j$ for all but precisely one $j$, 
i.e. $v_j\in V_1$ if and only if $v_j\in V_2$, except for exactly one $j$. 
Thus, $V_1\subset V_2$ or $V_2\subset V_1$, and the quotient is 1-dimensional, as desired.  
\end{proof}

\end{proof} 

For any parabolic subgroup $K$ in $G$, let $N(K)$ be the normalizer of
$K$ in $G$, and $N^0(K) := N(K)/K$. Let $i_K: V^K/N^0(K)\to V/G$ be
the corresponding closed embedding.

\begin{corollary}\label{quotsin1}
  There is a canonical isomorphism
  \[
  M(V/G) \cong \bigoplus_{K\in
    \Par(G,V)/G}HP_0(\OO_{(V^K)^\perp/K})\otimes
  (i_K)_*(IC(V^K/N^0(K))). 
  \]
\end{corollary}

\begin{proof} 
  By definition, $M(V/G) = \phi^G_* M_\phi(V)$. Thus, the corollary
  follows from Theorem \ref{quotsin} and Corollary \ref{pushpois},
  using that $\pi_*^G(\Omega_X)=IC(X/G)$ for every smooth
  $G$-variety $X$.
\end{proof}

\subsection{Generalization to the non-linear case}

We now generalize the result of the previous subsection to the case
when $V$ is a smooth connected symplectic variety with a faithful
action of a finite group $G$.

Let $T$ be a symplectic representation of a finite group $K$.  Denote
by $H(T)$ the $K$-representation
\[
H(T):=\OO_T/\lbrace{\OO_{T/K},\OO_T\rbrace}.
\]
By \cite[\S 7]{BEG} (or our Corollary \ref{co2} with $X=T$), $H(T)$ is
finite-dimensional.

Now let $E$ be a symplectic vector bundle on a connected variety $Y$.
Assume that $E$ carries a fiberwise faithful symplectic action of a
finite group $K$.  In this case, for every $y\in Y$, we can define the
vector space $H(E_y)$ as above. All these spaces are isomorphic, as
$K$-modules, to $H(T)$, where $T$ is a certain symplectic
representation of $K$. Since $H(T)$ is finite-dimensional, $H(E)$ is
an algebraic vector bundle on $Y$.

\begin{proposition}\label{flatco}
  The bundle $H(E)$ carries a canonical flat algebraic connection
  $\nabla$.  This connection defines an $\OO$-coherent $D$-module with
  regular singularities.
\end{proposition}
 
\begin{proof}
  Let us first explain what $\nabla$ is in topological terms, assuming
  $\mathbb K=\mathbb C$. If $y_0$ and $y_1$ are any two points of $Y$ and
  $y_t$, for $t\in [0,1]$, is any path from $y_0$ to $y_1$, pick a
  continuous family $g(t)$ of isomorphisms of $K$-modules
  $E_{y_0}\to E_{y_t}$, with $g(0)=\operatorname{Id}$, and define the
  transport of $v\in H(E_{y_0})$ along $y_t$ to be $g(1)v$.

  We claim that this is well-defined. Indeed, if $g_1(t)$ and $g_2(t)$
  are two such families, then $b(t):=g_1(t)^{-1}g_2(t)$ belongs to the
  group $L=Sp(E_{y_0})^K$, and hence to the connected component $L_0$
  of the identity in this group. However,
  $\mathrm{Lie}(L_0)\subset \OO_{E_{y_0}}^K$, so $b(t)v=v$, and thus
  $g(1)v$ does not depend on the choice of $g$.

  If $y_0$ and $y_1$ are infinitesimally close, this makes sense
  algebraically, and defines a flat algebraic connection $\nabla$ on
  $H(E)$.  It is easy to see that this connection has regular
  singularities, as its sections are algebraic.
\end{proof}

By the Riemann-Hilbert correspondence, if $\mathbb K=\mathbb C$, the
algebraic $D$-module $(H(E),\nabla)$ is determined by the monodromy of
the connection $\nabla$. This monodromy is the composition of the maps
$\rho: \pi_1(Y,y_0)\to \pi_0(L)=L/L_0$ and
$\theta: L/L_0\to GL(H(E_{y_0}))=GL(H(T))$. The map $\theta$ depends
only on the isomorphism class of the representation $T$.

Let us describe $\rho$.  Let $R_r, R_q$, and $R_c$ be the sets of
irreducible (complex) representations of $K$ of real, quaternionic,
and complex type modulo dualization, respectively. These correspond to the irreducible symplectic representations of $K$, via tensoring with the symplectic
vector space $\KK^2$, equipping with the canonical symplectic form, and sending a pair of nonisomorphic dual representations $Q, Q^*$ to the
space $Q \oplus Q^*$ with the standard symplectic pairing, respectively.  Let
$T_Q:=\Hom_{K}(Q,T)$ (over $\KK$, for an arbitrary representation $Q$). Then,
\[
L=\prod_{Q\in R_r}Sp(T_Q)\times \prod_{Q\in R_q}O(T_Q)\times
\prod_{Q\in \overline{R_c}}GL(T_Q),
\]
and hence,
\[
L/L_0=\prod_{Q\in R_q: T_Q\ne 0}\mathbb Z_2.
\]

Now, for any quaternionic representation $Q$ of $K$ which occurs in
$T$, $E_Q=\Hom_K(Q,E)$ is an orthogonal vector bundle on
$Y$.

\begin{proposition}
  \label{sw} The $Q$-th coordinate of $\rho$ is the first
  Stiefel-Whitney class
  $w_1(E_Q)\in H^1(Y,\mathbb Z_2)=\Hom(\pi_1(Y,y_0),\mathbb Z_2)$.
\end{proposition}

\begin{proof} 
  This follows immediately from the definition of the first
  Stiefel-Whitney class.
\end{proof}

Now consider $V,G$ as above, and $K\in \Par(G,V)$.  Note that the
connected components of the fixed point variety $V^K$ are
symplectic. Denote by $C_K$ the set of connected components of $V^K$.
For each $Z\in C_K$, let $i_Z: Z\to V$ be the corresponding closed
embedding.

Then, we obtain the following nonlinear generalization of Theorem
\ref{quotsin}, whose proof is parallel to the linear case and omitted.
\begin{theorem}\label{quotsin2}
  There is a canonical $G$-equivariant isomorphism
  \[
  M_\phi(V) \cong \bigoplus_{K\in \Par(G,V)}\bigoplus_{Z\in
    C_K}(i_Z)_*(H(TV|_{Z}/TZ)).
  \]
\end{theorem}

\begin{corollary}\label{poishom}
  Let $d_Z$ be the dimension of $Z$. Then,
  \[
  \OO_V/\lbrace{\OO_{V/G},\OO_V\rbrace}=
  \bigoplus_{K\in \Par(G,V)}\bigoplus_{Z\in C_K}
  H^{d_Z}(Z,H(TV|_{Z}/TZ)),
  \]
  and hence
  \[
  HP_0(\OO_{V/G})=\Bigl(\bigoplus_{K\in \Par(G,V)}\, \bigoplus_{Z\in
    C_K} H^{d_Z}(Z,H(TV|_{Z}/TZ))\Bigr)^G.
  \]
\end{corollary}

We also formulate a nonlinear generalization of Corollary
\ref{quotsin1}, whose proof (which we omit) is parallel to the linear
case.  To do so, let $N(K)$ be the normalizer of $K$ in $G$, and for
each $Z\in C_K$, denote by $N_Z(K)$ the subgroup of elements of $N(K)$
that map $Z$ to itself. Let $N_Z^0(K):=N_Z(K)/K$.  Let
$\pi_Z: Z\to Z/N_Z^0(K)$ be the corresponding projection and
$i_Z: Z/N_Z^0(K)\to V/G$ be the corresponding closed embedding.  Let
$H_Z=(\pi_Z)_*^{N_Z^0(K)}(H(TV|_{Z}/TZ))$ be the equivariant
pushforward of the $N_Z^0(K)$-equivariant local system $H(TV|_{Z}/TZ)$
from $Z$ to $Z/N_Z^0(K)$.

\begin{theorem}\label{quotsin3}
  There is a canonical isomorphism
  \[
  M(V/G)\cong\bigoplus_{K\in \Par(G,V)/G}\bigoplus_{Z\in
    C_K/N(K)}(i_Z)_*(H_Z).
  \]
\end{theorem}
The proof is similar to that of Theorem \ref{quotsin1} and is omitted.

\subsection{The $C^\infty$ case}

Corollary \ref{poishom} can be generalized to the complex analytic and
$C^\infty$-settings.  Moreover, in the $C^\infty$-setting, the
statement simplifies due to the following lemma from elementary
representation theory.

\begin{lemma}\label{tri}
  Let $T$ be a real symplectic representation of a finite group $K$.
  Then, the centralizer $L:=Sp(T)^K$ is connected.
\end{lemma}

\begin{proof} 
  Let $R_r^{\mathrm{re}}, R_c^{\mathrm{re}}$, and $R_q^{\mathrm{re}}$
  be the sets of irreducible real representations of $K$ of real,
  complex, and quaternionic type, respectively (i.e., with
  endomorphism algebras $\mathbb R$, $\mathbb C$, and $\mathbb H$).  That is,
  $Q\in R_r^{\mathrm{re}}$ if and only if
  $Q_{\mathbb C} := Q \otimes_{\mathbb R} \mathbb C$ is irreducible,
  $Q\in R_c^{\mathrm{re}}$ if and only if $Q_{\mathbb C}$ is the sum of
  two non-isomorphic complex representations of $K$ which are dual to
  each other, and $Q\in R_q^{\mathrm{re}}$ if and only if $Q_{\mathbb C}$
  is the direct sum of two copies of an irreducible complex
  representation of $K$.

  Let $T_Q:=\Hom_K(Q,T)$. Then, for $Q$ of real, complex, or
  quaternionic type, $T_Q$ is a vector space over $\mathbb R,\mathbb C$, or
  $\mathbb H$, respectively, which has a natural nondegenerate
  skew-Hermitian form with values in the corresponding division
  algebra.  Moreover, the group $Sp(T)^K$ is the product over $Q$ of
  the groups of linear transformations of $T_Q$ (over $\mathbb R$,
  $\mathbb C$, or $\mathbb H$) which preserve the skew-Hermitian form.

  Let the dimension of $T_Q$ over the corresponding division algebra
  be $d=d_Q$.  Then, the group of symmetries of the skew-Hermitian
  form is $Sp(d)$ in the real case, $U(p,q)$ for some $p=p_Q$ and
  $q=q_Q$ with $p+q=d$ in the complex case, and $O^*(2d)$ in the
  quaternionic case (see \cite{He}).  Thus,
\[
L=\prod_{Q\in R_r^{\mathrm{re}}}Sp(d_Q)\times \prod_{Q\in
  R_c^{\mathrm{re}}}U(p_Q,q_Q)\times \prod_{Q\in
  R_q^{\mathrm{re}}}O^*(2d_Q).
\]
The maximal compact subgroups of $Sp(d), U(p,q)$, and $O^*(2d)$ are
$U(d/2), U(p)\times U(q)$, and $U(d)$, respectively; they are
connected, so $L$ is a connected group, as desired.
\end{proof} 

Let $V$ be a compact symplectic $C^\infty$-manifold and $G$ be a
finite group acting faithfully on $V$ by symplectic
diffeomorphisms. Lemma \ref{tri} implies (using the notation of the
previous subsection) that for any $Z\in C_K$, the local system
$z\to H(T_zV/T_zZ)$ on $Z$ is trivial\footnote{In the complex
  algebraic or analytic case, the local system $H(T_zV/T_zZ)$ need not
  be trivial; an example will be discussed in a forthcoming paper.}.
Note that the fibers $H(T_z V/T_zZ)$ should now be defined using
$C^\infty$-functions, but the result is the same as in the case of
polynomial functions, so they are in particular finite-dimensional.
Denote by $H_Z$ the space of sections of this local system; it is
naturally identified with every fiber.  The space $H_Z$ carries a
natural action of the group $N_Z^0(K)$.

\begin{proposition}
  Let ${\mathcal D}$ be the space of distributions on $V$ that are
  invariant under Hamiltonian flow produced by $G$-invariant
  Hamiltonians. Then,
\begin{enumerate}
\item[(i)]
  $\displaystyle {\mathcal D}=\bigoplus_{K\in Par(G,V)}\bigoplus_{Z\in C_K}H_Z^*$, \\
  In particular, the space ${\mathcal D}$ is finite-dimensional, and
  \[
  \dim {\mathcal D}=\sum_{K\in Par(G,V)}\sum_{Z\in C_K}\dim H_Z.
  \]
\item[(ii)] The space ${\mathcal D}^G$ of invariant distributions on
  the symplectic orbifold $V/G$ is finite-dimensional, and is given by
\[
{\mathcal D}^G=\bigoplus_{K\in Par(G,V)/G}\bigoplus_{Z\in C_K/N(K)} (H_Z^*)^{N_Z^0(K)}.
\]
\end{enumerate}
\end{proposition}

\begin{proof}
  Let $T$ be a symplectic real vector space, $K\subset Sp(T)$.  By
  Proposition \ref{invdi1}, the space of distributions on $T$
  supported at $0$ and invariant under $K$-invariant Hamiltonian flow
  is naturally identified with $H(T)^*$.  Moreover, the same is true
  about the space of such distributions on $T\times X$ which are
  supported at $0\times X$, where $X$ is any connected symplectic
  manifold.  Now, by an equivariant version of the Darboux theorem,
  for every $K\in \Par(G,V)$, $Z\in C_K$, and $z\in Z$, $V$ can be
  represented near $z$ in the form $T\times X$ in a $K$-compatible way
  (where $X$ is an open set in the standard symplectic vector space of
  the appropriate dimension). Gluing together such representations
  along $Z$, we see that every element $\xi\in {\mathcal D}$ can be
  canonically written as a sum of contributions $\xi_Z$ supported on
  $Z\in C_K$, for $K\in \Par(G,V)$, and each such contribution belongs to
  $H^0(Z, H(TV|_{Z}/TZ)^*)$.  By Lemma \ref{tri}, the local systems
  $H(TV|_{Z}/TZ)$ are trivial.  This proves (i). Statement (ii)
  follows from (i) by taking $G$-invariants.
\end{proof}

\newpage

\appendix
\section{Prime and primitive ideals}

\vskip .1in
\centerline{\bf By Ivan Losev}
\vskip .1in

Let $A$ be an associative unital algebra equipped with an increasing exhaustive filtration
$\F_i A, i\geqslant 0$.
Consider the associated graded algebra $A_0$ of $A$ and the center $Z$ of $A_0$.
%Suppose that there is $d>0$ such that $[\F_i A, \F_j A]\subset \F_{i+j-d} A$.
Then $Z$ is a graded commutative algebra. Let $Z^i$ denote the $i$-th graded subspace.

We assume that there is an integer $d>0$ and a linear embedding $g:Z\rightarrow A$ satisfying the following conditions
\begin{itemize}
\item $g(Z^i)\subset \F_i A$ for all $i$ and, moreover, the composition of $g$ with the projection
$\F_i A\twoheadrightarrow \F_i A/\F_{i-1} A$ coincides with the inclusion $Z^i\hookrightarrow \F_i A/\F_{i-1} A$.
\item $[g(Z^i),\F_j A]\subset \F_{i+j-d}A$ for all $i,j$.
\end{itemize}

Under this assumption the algebra $Z$ has a natural Poisson bracket: for $a\in Z^i, b\in Z^j$ for
$\{a,b\}$ take the image of $[g(a),g(b)]$ in $F_{i+j-d}A/\F_{i+j-d-1} A$ (this image lies in $Z^{i+j-d}$).
The claim that the map $(a,b)\mapsto \{a,b\}$ extends to a Poisson bracket on $Z$ is checked directly.
It is also checked directly that this Poisson bracket is independent on the choice of the embedding $g$. 

For example, if $A$ is {\it almost commutative}, that is $Z=A_0$, then for $d$ one can take the maximal integer
with $[\F_i A, \F_j A]\subset \F_{i+j-d}A$ for all $i,j$ (and for $g$ the direct sum of sections
of the projections $\F_i A\twoheadrightarrow \F_i A/\F_{i-1}A=Z^i $).

Another example, where $d$ and $g$ exist is that of symplectic reflection algebras (see Example 1.5), there
$d=2$.

The goal of this appendix is to prove the following claim.

\begin{theorem}
Suppose that there are $g:Z\rightarrow A$ and $d>0$ as above. 
Further, suppose that
\begin{itemize} \item $Z$ is finitely generated and its spectrum
has finitely many symplectic leaves.
\item $A_0$ is a finitely generated  module over $Z$.
\end{itemize}
Then the following assertions hold
\begin{enumerate}
\item
$A$ has finitely many prime ideals,
\item every prime ideal is primitive.
\end{enumerate}
\end{theorem}

%Below it will be convenient for us to replace the filtration
%$\F_iA$ with a new filtration $\F'_iA$ defined by $\F'_iA:=\F'_{[i/2]}A$. This procedure does not change
%$A_0$ as an algebra but modifies ("doubles") the grading on $A_0$.

We start by deriving part (2) from part (1).  It is enough show that every prime ideal of $A$
is the intersection of primitive ideals containing it
(if this is the case one says that $A$ is a {\it Jacobson} ring).

Consider the {\it Rees algebra} $A_\hbar$ of $A$ defined by $A_\hbar:=\bigoplus_{i\geqslant 0} (\F_i A)\hbar^i\subset A[\hbar]$. This is a graded algebra. The quotient $A_\hbar/\hbar A_\hbar$ 
is naturally identified with $A_0$, while for $a\neq 0$ we have $A_\hbar/(\hbar-a)A_\hbar\cong A$.

Consider the algebra $A_\hbar/\hbar^{d}A_\hbar$ and let $Z_{(d)}$ denote its center. The existence of $g$
implies that the image of $Z_{(d)}$ in $A_0$ under the natural epimorphism $A_\hbar/ \hbar^{d}A_\hbar\twoheadrightarrow A_0$ coincides with $Z$. Let $Z_\hbar$ denote the preimage of $Z_{(d)}$ in $A_\hbar$.
Then $Z_\hbar$ is a graded subalgebra of
$A_\hbar$ and $\widetilde{A}_0:=Z_\hbar/\hbar Z_\hbar$ is commutative.
Moreover, the kernel of the natural epimorphism $\widetilde{A}_0\rightarrow Z$ is naturally identified
with $A_0/Z$. The square of the kernel is zero. So the kernel is a finitely generated $Z$-module.
It follows that $\widetilde{A}_0$ is  finitely generated. Now we remark that $Z_\hbar$
is flat and hence free as a graded $\K[\hbar]$-module. %We claim that $Z_\hbar$ is Jacobson.
The algebra $\widetilde{A}:=Z_\hbar/(\hbar -1)Z_\hbar$ is equipped with a filtration induced from the grading
on $Z_\hbar$. Its associated graded is $\widetilde{A}_0$.  
So we can apply Quillen's lemma, see e.g., \cite{MCR}, 9.7.3, to $\widetilde{A}[x]$ to get that $\widetilde{A}[x]$ acts by a character on any irreducible module.
The   algebra $\widetilde{A}[x]$ is $\Z_{\geqslant 0}$-filtered
($\K[x]$ has degree 0) and $\gr \widetilde{A}[x]=\widetilde{A}_0\otimes \K[x]$ is finitely generated.
Applying \cite{MCR}, Proposition 1.6 and Lemma 1.2, we see that $\widetilde{A}$ is Jacobson.
But the embedding $Z_\hbar\hookrightarrow A_\hbar$ gives rise to a homomorphism $\widetilde{A}\rightarrow A$
and this homomorphism is surjective. Being a quotient of a Jacobson ring, $A$ is Jacobson too.

%Now  \cite{Dixmier}, the proof of Proposition 3.1.15,  shows that
%any prime ideal $I$ is the intersection of primitive ones. Since we have only finitely many primitive
%ideals (any primitive ideal is prime), the intersection is finite. It follows that
%$I$ coincides with one of the primitive ideals.

Now let us prove  part (1). We will do this in seven steps.

{\it Step 1.}
 To a two-sided ideal
$I\subset A$ assign the ideal $R_\hbar(I):=\bigoplus_{i\geqslant 0} (I\cap \F_iA)\hbar^i$ in $A_\hbar$. This ideal is
homogeneous and $\hbar$-saturated in the sense that $\hbar a\in R_\hbar(I)$ implies $a\in R_\hbar(I)$.
Of course, an ideal $I_\hbar\subset A_\hbar$ is $\hbar$-saturated if and only if $A_\hbar/I_\hbar$ is $\K[\hbar]$-flat.
Conversely, to a homogeneous $\hbar$-saturated ideal $I_\hbar\subset A_\hbar$ assign
the ideal $I_\hbar/(\hbar-1)I_\hbar$ in $A$. It is easy to check that the maps $I\mapsto R_\hbar(I), I_\hbar\mapsto
I_\hbar/(\hbar-1)I_\hbar$ are mutually inverse bijections. Moreover, the ideal $R_\hbar(I)$ is prime if and only if
the ideal $I$ is prime. So we need to check that there are finitely many homogeneous $\hbar$-saturated prime
ideals in $A_\hbar$.

Recall that to any finitely generated module $M$ over a Noetherian
algebra one can assign its Gelfand-Kirillov dimension $\Dim M$. Now to a finitely generated
$A_\hbar$-module $M_\hbar$ one assigns its associated variety $\VA(M_\hbar)\subset \Spec(Z)$. By definition,
$\VA(M_\hbar)$ is the support of $M_\hbar/ \hbar M_\hbar$. We can define the associated
variety of a finitely generated $A$-module $M$ by equipping $M$ with a filtration and setting
$\VA(M)=\VA(R_\hbar(M))$. We have $\dim \VA(M)=\Dim M, \dim \VA(M_\hbar)=\Dim M_\hbar-1$.
%Recall that to any finitely generated $A$-module $M$ we can assign its Gelfand-Kirillov dimension
%$\Dim(M)$ and the associated variety $\VA(M)\subset X$. Then $\Dim(M)=\dim \VA(M)$.
%It is easy to see that $\Dim A_\hbar/R_\hbar(I)=\Dim A/I+1$. Similarly to a finitely generated  $A_\hbar $
%we can assign its associated  variety $\VA(A_\hbar/I_\hbar)$ in $Z$: the support of the $Z$-module $A_\hbar/(\hbar %A_\hbar+I_\hbar)$.
%Then $\Dim(A_\hbar/I_\hbar)=\dim \VA(A_\hbar/I_\hbar)+1$. We also remark that $\VA(A_\hbar/I_\hbar)$ is a Poisson subvariety
%in $A_\hbar$ provided $I_\hbar$ is $\hbar$-saturated.

Let $Y$ be a symplectic  leaf on $\Spec(Z)$. Consider the set $\Prim_Y$ consisting of all
homogeneous $\hbar$-saturated prime ideals  $I_\hbar\subset A_\hbar$ such that $\overline{Y}$ is an
irreducible component of maximal dimension in the associated variety $\VA(A_\hbar/I_\hbar)$. Since there are only
finitely many symplectic leaves, it is enough to show that each set $\Prim_Y$ is finite.

{\it Step 2.}
We are going to compare the set $\Prim_Y$ with the set of  prime ideals in a certain
completion of $A_\hbar$. A similar technique was used in \cite{Losqsa}.

Pick a point $y\in Y$ and let $\m_y$  be its maximal ideal in $Z$. Let $\widetilde{\m}_y$
denote the preimage of $\m_y$ in $Z_\hbar$ under the natural epimorphism $Z_\hbar\twoheadrightarrow Z$.
Consider the completion $A^\wedge_\hbar:=\varprojlim_k A_\hbar/A_\hbar\widetilde{\m}_y^k$.
Since $[A_\hbar,Z_\hbar]\subset \hbar^d A_\hbar$, we see that $A_\hbar \widetilde{\m}_y^k$
is actually a two-sided ideal coinciding with $\widetilde{\m}_y^k A_\hbar$.
Therefore $A^\wedge_\hbar$ acquires a natural
complete topological algebra structure. Moreover, $(A_\hbar\widetilde{\m}_y^{k})(A_\hbar\widetilde{\m}_y^l)=A_\hbar\widetilde{\m}^{k+l}$.
One can easily show that $A_\hbar^\wedge/\hbar A_\hbar^\wedge=A_0^\wedge:=\varprojlim_k A_0/A_0\m_y^k$.
The algebra $A_0^\wedge$ is finite over the completion $\varprojlim_k Z/\m_y^k$. The latter is the completion
of a commutative finitely generated algebra and hence is Noetherian. Therefore $A_0^\wedge$ is also Noetherian.
It follows that the algebra $A_\hbar^\wedge$ is Noetherian too.

%Написать.

For a finitely generated $A_\hbar$-module $M_\hbar$ consider the completion $M_\hbar^\wedge:=\varprojlim_{k\rightarrow\infty} M_\hbar/ \widetilde{\m}^k_y M_\hbar$.
To ensure good properties of completions one needs an analog of the Artin-Rees lemma
from Commutative algebra, see \cite{Eisca}, Chapters 5 and 7.
The following lemma incorporates all results related to the Artin-Rees lemma we need.

\begin{lemma}\label{Lem:Rees}
\begin{enumerate}
\item The blow-up algebra $\bigoplus_{i\geqslant 0} A_\hbar \widetilde{\m}^{i}_y$ is Noetherian.
\item The Artin-Rees lemma holds for any submodule $N_\hbar$ in a finitely generated  left $A_\hbar$-module $M_\hbar$. That is,
there exists $k\in \N$ such that $N_\hbar \cap \widetilde{\m}_{y}^{k+l}M_\hbar=\widetilde{\m}_y^l(N_\hbar\cap \widetilde{\m}^k_y M_\hbar)$ for all $l\geqslant 0$.
\item The completion functor $M_\hbar\mapsto M_\hbar^\wedge$ is exact and isomorphic to
$M_\hbar\mapsto A_\hbar^\wedge\otimes_{A_\hbar}M_\hbar$.
\item If $M_\hbar$ is $\K[\hbar]$-flat, then $M^\wedge_\hbar$ is $\K[[\hbar]]$-flat.
\item The quotient $M_\hbar^\wedge/\hbar M_\hbar^\wedge$ coincides with the completion of
the $Z$-module $M_\hbar/\hbar M_\hbar$ at $y$.
\item $M_\hbar^\wedge=\{0\}$ if and only if $y\not\in\VA(M_\hbar)$.
\item Any finitely generated $A^\wedge_\hbar$-module is complete and separated in the
$\widetilde{\m}_y$-adic topology.
\item Any submodule in a finitely generated $A^\wedge_\hbar$-module is closed in the
$\widetilde{\m}_y$-adic topology.
\end{enumerate}
\end{lemma}
\begin{proof}
Let us prove (1). The remaining statements are more or less standard corollaries of (1) and are proved analogously to the corresponding statements in Subsection 2.4 of \cite{Losfdr} (the Artin-Rees lemma (assertion 2) is not proved there explicitly). The algebra $\bigoplus_{i\geqslant 0} A_\hbar \widetilde{\m}^i_y$ is generated by $A_\hbar$ (=the component of degree 0) as a module over the blow-up algebra $\bigoplus_{i\geqslant 0}\widetilde{\m}^i_y$. Since $A_\hbar$
is a finite $Z_\hbar$-module, we see that $\bigoplus_{i\geqslant 0} A_\hbar \widetilde{\m}^i_y$ is finite
over $\bigoplus_{i\geqslant 0}\widetilde{\m}^i_y$.

The algebra $Z_\hbar/\hbar Z_\hbar$ is commutative and finitely generated
(see the proof of the implication (2)$\Rightarrow$(1) above).
From here it is easy to deduce that $Z_\hbar$ is Noetherian.
So $\widetilde{\m}_y$ is a finitely generated $Z_\hbar$-module.
It follows that $\bigoplus_{i\geqslant 0}\widetilde{\m}^i_y$
is a finitely generated module over $\bigoplus_{i\geqslant 0}\widetilde{\m}^{2i}_y$. Indeed, the former
is generated by $\widetilde{\m}_y$ (its component of degree 1) over the latter.

 Therefore it is enough to check that
$\bigoplus_{i\geqslant 0}\widetilde{\m}^{2i}_y$ is Noetherian.
This will follow if we check that the pair $\widetilde{\m}_y^2\subset Z_\hbar$ satisfies the assumptions
of Lemma 2.4.2 in \cite{Losfdr}.
We have already checked that $Z_\hbar/\hbar Z_\hbar$ is commutative and Neotherian. Also  $[\widetilde{\m}^2_y,\widetilde{\m}^2_y]\subset \hbar \widetilde{\m}^2_y$ because $[\widetilde{\m}_y,\widetilde{\m}_y]= [Z_\hbar,Z_\hbar]\subset \hbar^d Z_\hbar$.
\end{proof}
%A little bit more subtle fact is that $R^\wedge_\hbar(A)$ is flat over $\K[[\hbar]]$, which was
%essentially verified in the proof of Lemma 2.4.1 in \cite{Losfdr}.

Now let us  construct a certain derivation $D$ of $A_\hbar^\wedge$. Consider  the derivation $D$ of $A_\hbar$ induced by the grading: $D(a)=ka$ for $a\in A_\hbar$
of degree $k$. Then $D$ is continuous in the $\widetilde{\m}_y$-adic topology so we can extend
it to $A_\hbar^\wedge$, the extension is the derivation we need.

Let $I_\hbar$ be a homogeneous ideal in $A_\hbar$. Then its completion $\overline{I}_\hbar$ with
respect to the $\widetilde{\m}_y$-adic topology is a $D$-stable ideal in $A^\wedge_\hbar$.
Assertions (3) and (4) of Lemma \ref{Lem:Rees} imply that $I^\wedge_\hbar$ is $\hbar$-saturated whenever $I_\hbar$ is.
By assertion (7), $I^\wedge_\hbar$ is closed in $A_\hbar^\wedge$ with respect to the $A\widetilde{\m}_y$-adic topology.

To a two-sided ideal $J_\hbar\subset A^\wedge_\hbar$ we assign its inverse image $J_\hbar^{fin}$ under the natural homomorphism $A_\hbar\rightarrow A^\wedge_\hbar$.  Clearly, $J_\hbar^{fin}$ is  homogeneous and $\hbar$-saturated
provided $J_\hbar$ is so.

{\it Step 3.} Now we would like to decompose $A^\wedge_\hbar$ into the completed tensor
product of a Weyl algebra and of some "slice" algebra. Let $\widetilde{X}$ denote the spectrum of
$Z_\hbar/\hbar Z_\hbar$. This is a non-reduced Poisson scheme, the corresponding reduced scheme is $X$.

Consider the completions $\widetilde{X}^\wedge,Y^\wedge$ of $\widetilde{X},Y$ at $y$. Then $Y^\wedge$ is a symplectic
leaf of the Poisson formal  scheme $\widetilde{X}^\wedge$. According to \cite{Kalss}, Proposition 3.3,
$\widetilde{X}^\wedge$ can be decomposed into the product $Y^\wedge\times \underline{X}^\wedge$, where
$\underline{X}^\wedge$ is a Poisson formal scheme such that a point $y\in \underline{X}^\wedge$
forms a  symplectic leaf. We are going to show that one can lift  this decomposition   to a decomposition of $A$.

More precisely, let $\W^\wedge_\hbar$  denote the completed Weyl algebra of $T_yY$, i.e., the algebra $\K[[T_yY,\hbar]]$ of formal power series in $T_yY$ and $\hbar$, where the multiplication is given by the Moyal-Weyl star-product $f*g=\mu\exp(\frac{\hbar^d P}{2} f\otimes g)$. Here $\mu$ denotes the multiplication map, $P$ is a Poisson bivector on $T_yY$.
We claim  there is a complete topological subalgebra $\underline{A}^\wedge_\hbar\subset A^\wedge_\hbar$
such that $\Spec(\underline{A}^\wedge_\hbar/(\hbar))=\underline{X}^\wedge_x$ and
\begin{equation}\label{eq:A1} A^\wedge_\hbar=
\W^\wedge_\hbar\widehat{\otimes}_{\K[[\hbar]]}\underline{A}_\hbar^\wedge,
\end{equation} where $\widehat{\otimes}$
stands for the completed (in the $\widetilde{\m}_y$-adic topology) tensor product.

%For brevity, set $V:=T_yY$. The  decomposition $$

Similarly to the proof of Proposition 7.1 in \cite{Los1dr}, we see that the embedding $\K[Y]^\wedge_y\hookrightarrow
\K[X]^\wedge_y$ lifts to an embedding $\W^\wedge_\hbar\hookrightarrow Z_\hbar^\wedge\subset A_\hbar^\wedge$.
For $\underline{A}_\hbar^\wedge$ take the centralizer of $\W^\wedge_\hbar$ in $A_\hbar^\wedge$. We have a natural topological algebra homomorphism $\W^\wedge_\hbar\widehat{\otimes}_{\K[[\hbar]]}\underline{A}^\wedge_\hbar$.
We need to check that it is an isomorphism. This is done analogously to the proof of Proposition 3.3.1 from
\cite{Losfdr}.

%The algebra $\underline{A}^\wedge_\hbar$ is the "slice" algebra we need.
In the sequel, we identify $A^\wedge_\hbar$ with the completed tensor product
above.

{\it Step 4.} We want to construct a bijection between the set of $\hbar$-saturated two-sided ideals
in $A^\wedge_\hbar$ and the analogous set for $\underline{A}^\wedge_\hbar$.
It is easy to see (compare with Lemma 3.4.3 in \cite{Losqsa}) that any $\hbar$-saturated two-sided
ideal $J_\hbar\subset A^\wedge_\hbar$ (which is automatically closed, see Lemma \ref{Lem:A2})
is of the form $\W^\wedge_\hbar\widehat{\otimes}_{\K[[\hbar]]}\underline{J}_\hbar$
for a uniquely determined $\hbar$-saturated two-sided ideal $\underline{J}_\hbar\subset \underline{A}^\wedge_\hbar$.
Also
\begin{equation}\label{eq:A3}
\VA(A^\wedge_\hbar/I^\wedge_\hbar)=Y^\wedge\times \VA(\underline{A}_\hbar^\wedge/\underline{I}^\wedge_\hbar).
\end{equation}

{\it Step 5.}
Recall that we are primarily interested in {\it $D$-stable} $\hbar$-saturated
ideals in $A^\wedge_\hbar$. So  we want to construct  a derivation
$\underline{D}$ of $\underline{A}^\wedge_\hbar$ such that $\underline{D}(\hbar)=\hbar$
and $\underline{\J}_\hbar$ is $\underline{D}$-stable provided $\J_\hbar$ is $D$-stable.

Pick a lagrangian subspace $U\subset T_yY$. Note that
\begin{equation}\label{eq:A2}
\underline{A}^\wedge_\hbar:=(A^\wedge_\hbar/A^\wedge_\hbar U)^{\ad U}.
\end{equation}

Pick a basis $u_1,\ldots,u_k\in U$. We claim that there is $a\in A^\wedge_\hbar$ such that
\begin{equation}\label{eq:A11}[u_1,a]=\hbar^d D(u_1).\end{equation} Indeed, choose a complimentary lagrangian subspace $U'\subset T_yY$
and pick the basis $v_1,\ldots,v_k\in U'$ dual to $u_1,\ldots,u_k$. There are unique elements $$P_i\in \K[[u_1,\ldots,u_k,v_2,\ldots, v_k,\hbar]]\widehat{\otimes}_{\K[[\hbar]]}\underline{A}^\wedge_\hbar$$
such that $D(u_1)=\sum_{i=0}^\infty
P_i v_1^i$.
Then  $a:=\int D(u_1) dv_1:= \sum_{i=0}^\infty P_i \frac{v_1^{i+1}}{i+1}$ satisfies (\ref{eq:A11}).

Now $D+\frac{1}{\hbar^d}\ad(a)$ induces a derivation $D^{(1)}$ of
$$\K[[u_2,\ldots,u_n,v_2,\ldots,v_n,\hbar]]\widehat{\otimes}_{\K[[\hbar]]}\underline{A}^\wedge_\hbar=
(A^\wedge_\hbar/A^\wedge_\hbar u_1)^{\ad u_1},$$
with $D^{(1)}(\hbar)=\hbar$.

Proceeding in the same way, we get a derivation  $\underline{D}'$ of $\underline{A}^\wedge_\hbar$
with $\underline{D}'(\hbar)=\hbar$.

Now suppose that $\underline{D}$ is a derivation of $\underline{A}^\wedge_\hbar$ such that
$D(\hbar)=\alpha \hbar,\alpha\in \K$. We can uniquely
extend $\underline{D}$ to $A^\wedge_\hbar$ by requiring that $\underline{D}$
acts on $T_yY$ by $\frac{\alpha d}{2}\operatorname{id}$.

\begin{lemma}\label{Lem:A5}
There is a derivation $\underline{D}$ of $\underline{A}^\wedge_\hbar$ such that $\underline{D}(\hbar)=\hbar$
and $D-\underline{D}=\frac{1}{\hbar^d}\ad(a)$ for some element $a\in A^\wedge_\hbar$.
\end{lemma}
\begin{proof}
Set $D_0:=D-\underline{D}'$. This is a $\K[[\hbar]]$-bilinear derivation of $A^\wedge_\hbar$.
Recall a basis $u_1,\ldots,u_k,v_1,\ldots,v_k\in T_yY$. Replacing $D_0$ with $D_0+\frac{1}{\hbar^d}\ad(\int D_0(u_1)dv_1)$ we get $D_0(u_1)=0$. Therefore $D_0(v_1)$ commutes with $u_1$. In other words,
$$D_0(v_1)\in \K[[u_1,\ldots,u_n,v_2,\ldots,v_n,\hbar]]\widehat{\otimes}_{\K[[\hbar]]}\underline{A}^\wedge_\hbar.$$
So replacing $D_0$ with $D_0-\frac{1}{\hbar^d}\ad(\int D_0(v_1)du_1)$ we get $D_0(u_1)=D_0(v_1)=0$. Proceeding
in the same way, we get $D_0|_{T_yY}=0$. Hence $D_0(\underline{A}^\wedge_\hbar)$ commutes
with $T_yY$. Therefore $D_0$ is induced by  a derivation
$\underline{D}_0$ of $\underline{A}^\wedge_\hbar$ such that $\underline{D}_0(\hbar)=0$.
We can put $\underline{D}:=\underline{D}'+\underline{D}_0$.
\end{proof}

We see that an $\hbar$-saturated ideal $J_\hbar$ is $D$-stable if and only if
$\underline{J}_\hbar$ is $\underline{D}$-stable.

{\it Step 6.} Let $\underline{\Prim}_{fin}$ denote the set of all prime $\underline{D}$-stable
$\hbar$-saturated ideals $\underline{J}_\hbar\subset \underline{A}^\wedge_\hbar$
such that the $\K[[\hbar]]$-module $\underline{A}_\hbar^\wedge/\underline{J}_\hbar$
is of finite rank.

Consider a map $\underline{\Prim}_{fin}\rightarrow \Prim_Y$ given by $\underline{J}_\hbar\mapsto J_\hbar^{fin}$, where,
as on Step 4, $J_\hbar:=\W^\wedge_\hbar\widehat{\otimes}_{\K[[\hbar]]}\underline{J}_\hbar$.
We claim that this map is surjective.

First of all, let us note that $\underline{J}_\hbar\mapsto J_\hbar$ gives
a bijection between $\underline{\Prim}$ and the set $\Prim^\wedge_Y$ of all prime $D$-stable
$\hbar$-saturated ideals $J_\hbar\subset A^\wedge_\hbar$. Now the claim that
the map in consideration is surjective stems from the following lemma.

\begin{lemma}\label{Lem:A2}
Let $I_\hbar\subset \Prim_Y$. Then any minimal prime ideal  $J_\hbar$ of $I^\wedge_\hbar$
lies in $\Prim^\wedge_Y$ and $J^{ fin}_\hbar=I_\hbar$.
\end{lemma}
\begin{proof}
Clearly, $\VA(A^\wedge_\hbar/J_\hbar)\subset \VA(A^\wedge_\hbar/I^\wedge_\hbar)$. The right hand side coincides
with $Y^\wedge$ by Lemma \ref{Lem:Rees}. On the other hand, $Y^\wedge$ is a  symplectic
leaf, so $\VA(A^\wedge_\hbar/J_\hbar)$ needs to coincide with $Y^\wedge$. Since $I^\wedge_\hbar$
is $D$-stable, so is every its minimal prime ideal.

We have
$I_\hbar\subset J_\hbar^{fin}$ whence $\VA(A_\hbar/J_\hbar^{fin})\subset \VA(A_\hbar/I_\hbar)$.
It follows that $\Dim(A_\hbar/I_\hbar)\geqslant \Dim (A_\hbar/J_\hbar^{fin})$.
On the other hand $Y\subset \VA(A_\hbar/J_\hbar^{fin})$ and $\dim \VA(A_\hbar/I_\hbar)=\dim Y$.
So $\Dim (A_\hbar/I_\hbar)=\Dim (A_\hbar/ J^{fin}_\hbar)$.  
Applying \cite{BK}, Corollar 3.6, we see that $I_\hbar=J^{fin}_\hbar$.
\end{proof}

{\it Step 7.} To complete the proof of assertion (1) of the theorem it is enough to check
that $\overline{\Prim}_{fin}$ is finite.
  By Corollary \ref{fintracepm},
$(\underline{A}^\wedge_\hbar/[\underline{A}^\wedge_\hbar,\underline{A}^\wedge_\hbar])[\hbar^{-1}]$
is finite-dimensional over $\K((\hbar))$. Hence, as in Theorem \ref{th2}, $\underline{A}^\wedge_\hbar[\hbar^{-1}]$ has finitely many finite-dimensional irreducible representations.  To each distinct ($\hbar$-saturated) ideal $\underline{J}_\hbar \in \underline{\Prim}_{fin}$ we can associate a distinct ideal $\underline{J}_\hbar[\hbar^{-1}] \subset \underline{A}^\wedge_\hbar[\hbar^{-1}]$. Each of these finite codimension ideals is the kernel of at least one
 irreducible finite-dimensional representation of $\underline{A}^\wedge_\hbar[\hbar^{-1}]$.  Thus, the set $\overline{\Prim}_{fin}$ must be itself finite.

{\bf Acknowledgements.} The author would like to thanks Pavel Etingof and Alexander Premet
for useful discussions.

\section{Possible generalizations}

In this second appendix, we briefly discuss some possible
generalizations of the preceding appendix which were suggested by
I. Losev (in particular, he suggested questions (1)--(3) below).

One can ask the following more general questions of $A$, which seem to
have positive answers in all known cases:
\begin{enumerate}
\item If $\dim A = \infty$, is it true that $A$ is not residually
finite-dimensional? I.e., is the intersection of the kernels of all finite-dimensional representations of $A$ nontrivial?
\item Is it true that the category $\Rep_{f.d.}(A)$ of finite-dimensional representations of $A$ is equivalent to the category $\Rep_{f.d.}(A')$ for a finite-dimensional algebra $A'$? Equivalently: (i) Does $\Rep_{f.d.}(A)$ have enough projectives?  (ii) Does $A$ have a minimal finite-codimensional ideal?
\item Is $A$ of finite length as an $A$-bimodule?
\end{enumerate}
We claim that (3) $\Rightarrow$ (2) $\Leftrightarrow$ (1). Moreover,
we expect that all three are equivalent in the more general
\emph{deformational} setting of Remark \ref{hbrem}.

First we show the implications $\Rightarrow$.  Indeed,
for a fixed algebra $A$, property (3) implies (2),
since if $A$ has finite length as an $A$-bimodule, then in
particular there is a bound on the codimension of any
finite-codimensional ideal, so the intersection of all
finite-codimensional ideals is the minimal such.  Next, again for a fixed
algebra $A$, (2) implies (1),
since if $A$ has a minimal ideal of finite codimension, then this must
lie in the kernel of every finite-dimensional representation.

Next, we show that (1) $\Rightarrow$ (2). Assume that (1) holds for
all algebras $A$ satisfying the assumptions of the appendix. Take such
an algebra $A$. We claim that the intersection of all
finite-codimensional ideals of $A$ has finite codimension.  Let $J$ be
this intersection.  Consider the new algebra $A' = A/J$.  Then, in
$A'$, the intersection of all finite-codimensional ideals is zero, i.e.,
$A'$ is residually finite-dimensional.  On the other hand, if $Z$ is
the center of $\gr A$, then $\gr A'$ is finite over $Z/(Z \cap \gr
J)$, and the latter is a Poisson subscheme of $\Spec Z$, whose
symplectic leaves are therefore a subset of the symplectic leaves of
$\Spec Z$ itself, and therefore there are finitely many.  Hence, by
(1), $A'$ is itself finite-dimensional, as desired.

Finally, let us consider the analogue of the above questions in the
setting of deformation quantizations $A_\hbar$ as in Remark
\ref{hbrem}. In this case, one is interested in ideals of
$A_\hbar[\hbar^{-1}]$ over $\KK((\hbar))$.  We expect that all three
questions are then equivalent (and specialize to the case of filtered
quantizations $A$ of a graded algebra $A_0$, by letting $A_\hbar$ be
the $\hbar$-adic completion of the Rees algebra of $A$).  Here, we
explain why (2) for such $A_\hbar$ implies (3) for filtered
quantizations. Fix $A_0$ and a filtered quantization $A$.  Let
$A_\hbar$ be the Rees algebra of $A$, and let $Z$ be the center of
$A_0$.  We claim that any strictly decreasing chain of ideals of
$A_\hbar[\hbar^{-1}]$ is finite.  We can replace $A_\hbar$ by its
completion $A^\wedge_\hbar$ at a closed point $y \in \Spec Z$.  By
\eqref{eq:A1}, it is enough to show the claim for the slice algebra
$\underline{A}^\wedge_\hbar$.  As above, every ideal of
$\underline{A}^\wedge_\hbar[\hbar^{-1}]$ has finite codimension over
$\KK((\hbar))$ (its associated graded has zero-dimensional support, at
the unique closed point of $\Spec
Z(\underline{A}^\wedge_\hbar[\hbar^{-1}])$ over
$\KK((\hbar))$). Hence, it follows from (2) that the chain is finite.

\bibliographystyle{amsalpha}
\bibliography{master}
\end{document}